\documentclass[notitlepage,reqno,12pt]{amsart}
\usepackage[left=1.5in,right=1.5in,top=1in,bottom=1.5in]{geometry}
\usepackage{titlesec}
\usepackage{graphicx} 
\usepackage{latexsym,amssymb,epsfig,amsmath,amsfonts,amsthm}
\usepackage{float}

\usepackage{rotating}
\usepackage[toc,page]{appendix}
\usepackage{color}
\usepackage{multirow}
\usepackage{relsize}
\usepackage{microtype}
\usepackage{cancel}
\usepackage[foot]{amsaddr}
\usepackage[colorlinks, allcolors=blue]{hyperref}
\usepackage{subcaption}
\usepackage{algpseudocode}
\usepackage{algorithm}
\usepackage[utf8]{inputenc}
\usepackage{enumitem}
\usepackage{fancyhdr}
\usepackage{tikz}
\usepackage{ulem}
\numberwithin{equation}{section}
\usepackage{mathrsfs}
\usepackage{accents}
\usepackage{bm}
\usepackage{mathtools}
\usepackage{cleveref}
\usepackage{array}

\theoremstyle{definition}
\newtheorem{definition}{Definition}
\theoremstyle{lemma}
\newtheorem{lemma}{Lemma}
\theoremstyle{theorem}

\theoremstyle{assumption}

\theoremstyle{remark}
\newtheorem{remark}{Remark}

\titleformat{\section}
{\normalfont\large\bfseries}{\thesection}{1em}{}

\titleformat{\subsection}
{\normalfont\normalsize\bfseries}{\thesubsection}{1em}{}
\numberwithin{subsection}{section}

\title{A robust first order meshfree method for\\time-dependent nonlinear conservation laws}
\author{Samuel Kwan, Jesse Chan}
\date{}

\begin{document}

\maketitle

\begin{abstract}
\noindent We introduce a robust first order accurate meshfree method to numerically solve time-dependent nonlinear conservation laws. The main contribution of this work is the meshfree construction of first order consistent summation by parts differentiations. We describe how to efficiently construct such operators on a point cloud. We then study the performance of such differentiations, and then combine these operators with a numerical flux-based formulation to approximate the solution of nonlinear conservation laws, with focus on the advection equation and the compressible Euler equations. We observe numerically that, while the resulting mesh-free differentiation operators are only $O(h^\frac{1}{2})$ accurate in the $L^2$ norm, they achieve $O(h)$ rates of convergence when applied to the numerical solution of PDEs.
\end{abstract}

\section{Introduction}

\noindent Numerical methods for solving partial differential equations (PDEs) form the backbone of computational modeling and simulation efforts in science and engineering. The majority of numerical methods for PDEs rely on a representation of the domain as a mesh. However, solution quality and mesh quality are strongly related, such that poor quality meshes with irregularly shaped elements result in poorly approximated solutions \cite{babuvska1976angle, shewchuk2002good, katz2011mesh}. This is especially problematic in 3D, where it is difficult to automatically and efficiently generate unstructured meshes with guaranteed element quality \cite{boggs2005dart}.  \\

\noindent Meshfree methods encompass a broad class of numerical schemes intended to circumvent the mesh generation step. These methods range from particle-based methods to high order collocation type schemes \cite{fornberg2015solving}. However, a common issue faced by meshfree discretizations is balancing accuracy with stability and robustness. While there are a variety of methods for constructing accurate structure preserving mesh-based discretiztations \cite{fjordholm2012arbitrarily, lipnikov2014mimetic, cotter2023compatible}, it is more difficult to ensure that meshfree methods are conservative and stable \cite{tominec2021least, glaubitz2023towards}. \\ 

\noindent In this paper, we present a method for constructing meshfree discretizations based on the enforcement of a summation-by-parts (SBP) property. This work is closely related to the formulation of  \cite{kwan2012conservative}, but exploits the fact that enforcing only first order accuracy constraints results in a simpler construction of meshfree operators \cite{trask2020conservative}. These operators are then used to formulate a meshfree semi-discretization in terms of finite volume fluxes, as is done in \cite{kwan2012conservative}. If these fluxes are local Lax-Friedrichs fluxes with appropriate wave-speed estimates, the resulting discretization can be shown to be invariant-domain preserving under forward Euler time-stepping and a CFL condition. \\

\noindent The paper proceeds as follows. In Section~\ref{sec:sbp}, we introduce the concept of summation by parts (SBP) operators. In Section~\ref{sec: construct SBP Ops}, we introduce a methodology for constructing SBP and norm matrices given a point cloud, an adjacency matrix, and surface information (e.g., outward unit normals and boundary quadrature weights). In Section~\ref{sec: finding adjacency}, we explore different ways of constructing the adjacency matrix and its effect on the behavior of the differentiation matrix. In Section~\ref{sec: numerical method}, we discuss how to construct a numerical method for nonlinear conservation laws based on the meshfree operators, and in Section~\ref{sec: results}, we present some numerical results where we apply the method to the 2D linear advection and compressible Euler equations.

\section{Summation by parts operators}
\label{sec:sbp}

    \noindent The goal of this paper is to create a first order accurate meshless numerical method to solve non-linear conservation laws. The main tool in building such a numerical method will be summation by parts differentiation operators. Such methods are useful for constructing numerical discretizations which are both conservative and satisfy an energetic or entropic statement of stability \cite{kwan2012conservative}. \\

    \noindent Borrowing from the notation in \cite{Main_Paper},  we first introduce summation by parts operators. Consider $u(x) \in L^2(\Omega)$ where $\Omega$ is our domain. We start with an arbitrary set of $N$ nodes $S = \{(x_i, y_i)\}^N_{i=1}$ where $N = N_{int} + N_{b}$ ($N_{int}$ refers to the number of interior points and $N_{b}$ refers to the number of points in $\partial \Omega$). 
    Hence $u$ on $S$ is denoted by:

    \newcommand{\fnt}[1]{\bm{\mathsf{ #1}}}
    
    \begin{equation}
        \fnt{u} = [u(x_1, y_1), ... , u(x_N, y_N)]^T
    \end{equation}
    and 
    \begin{equation}
        \fnt{u}_x \equiv \left[\frac{\partial u}{\partial x}(x_1, y_1), ..., \frac{\partial u}{\partial x}(x_N, y_N)\right]^T.
    \end{equation}

\noindent We begin by introducing some matrices which are used to approximate integrals. We first assume that we are given a matrix $\fnt{H}$ which is diagonal and positive definite ``norm'' matrix such that
\begin{equation}
\fnt{u}^T\fnt{H}\fnt{v} \approx \int_\Omega uv.
\label{eq:H_def}
\end{equation}
We will define the specific form of the matrix $\fnt{H}$ in Section \ref{sec: construct SBP Ops}. The choice of a diagonal (e.g., ``lumped'') norm matrix is made to simplify the construction of a robust numerical meshfree scheme \cite{guermond2019invariant}. 

\noindent Furthermore, we assume that we are given a matrix $\fnt{E}_x$ which satisfies:
\begin{equation}
    \fnt{u}^T \fnt{E}_x \fnt{v} \approx \int_{\partial \Omega} u v n_x. \label{prop3b}
\end{equation}

\noindent A matrix $\fnt{D}_x = \fnt{H}^{-1} \fnt{Q}_x$ is a first order accurate summation by parts (SBP) differentiation if it satisfies the following properties: \\
\begin{equation}  
        \fnt{Q}_x \fnt{1} = \fnt{0} \label{prop1}
\end{equation}
\begin{equation}
    \fnt{Q}_x = \fnt{S}_x + \frac{1}{2}\fnt{E}_x \; 
    \; \text{where $\fnt{S}_x$ is skew-symmetric and $\fnt{E}_x$ is symmetric} \label{prop3}
\end{equation}
\begin{equation}
    \fnt{1}^T \fnt{E}_x \fnt{1} = \int_{\partial \Omega}  n_x.  \label{prop2}
\end{equation}

\noindent Here, \eqref{prop1} is a consistency condition, which implies that $\fnt{D}_x$ exactly differentiates constants. The second property \eqref{prop3} is the summation-by-parts or SBP property.  \\

\noindent The SBP operator $\fnt{D}_x$ approximates the first derivative $\fnt{D}_x \fnt{u} \approx \fnt{u}_x$ while imitating integration by parts through the SBP property \eqref{prop3} and the integral approximations \eqref{eq:H_def} and \eqref{prop3b}. Consider two differentiable functions $u, v$. Integration by parts gives
\begin{equation}
    \int_{\Omega} v \frac{\partial u}{\partial x} + \int_{\Omega} u \frac{\partial v}{\partial x} = \int_{\partial \Omega} vun_x.
\end{equation}

\noindent \eqref{prop3} tells us that
\begin{equation}
    \fnt{Q}_x + \fnt{Q}_x^T = \fnt{E}_x \label{Q-E property}.
\end{equation} 
Combining \eqref{Q-E property} with $\fnt{D}_x = \fnt{H}^{-1}\fnt{Q}_x$ tells us that

\begin{equation}
    \fnt{v}^T\fnt{HD}_x\fnt{u} + \fnt{u}^T\fnt{HD}_x\fnt{v} = \fnt{v}^T\fnt{E}_x\fnt{u}.
\end{equation}

\noindent Here, the matrix $\fnt{Q}_x$ encodes approximations of the following integral:
\begin{align}
    \fnt{v}^T\fnt{Q}_x\fnt{u} \approx \int_{\Omega} v \frac{\partial u}{\partial x}. \label{int1} 
\end{align}

\noindent Similar relations hold for $\fnt{Q}_y, \fnt{E}_y$.

\section{Constructing meshfree SBP Operators} \label{sec: construct SBP Ops}

\noindent In order to construct meshfree SBP operators, we will first construct norm and boundary operators $\fnt{H}, \fnt{E}_x, \fnt{E}_y$, which will then be used to construct SBP differentiation operators $\fnt{Q}_x, \fnt{Q}_y$ which satisfy \eqref{prop1} - \eqref{prop3}.  Meshfree SBP operators have previously been constructed by solving a nonlinear optimization problem with accuracy-based constraints \cite{kwan2012conservative}. In this work, by assuming only a first order consistency constraint, we adapt the method of \cite{Inverting_L} to construct $\fnt{Q}_x, \fnt{Q}_y$. This approach relies only on the solution of a graph Laplacian matrix equation and simple algebraic operations. \\

\noindent We will outline our approach to constructing $\fnt{E}_x$ and $\fnt{Q}_x$ in the following sections. The procedure for $\fnt{E}_y$ and $\fnt{Q}_y$ will be identical. 

\subsection{Boundary operators} \label{subsec: constructing E}

\noindent Following \cite{Main_Paper}, $\fnt{E}_x$ is defined as a diagonal matrix:
\begin{equation}
   (E_x)_{ii}  = 
        \left\{
            \begin{array}{lr}
            w_i n_{x,i} & \text{if $i$ is a boundary node} \\
            0 & \text{otherwise} 
            \end{array}
        \right\},        
\end{equation}
\noindent where $w_i$ is a quadrature weight associated with the $i$th point on the domain boundary and $n_{x,i}$ denotes the value of the $x$ component of the outward normal at this point. For example, for the domains considered in this paper, the boundary is a collection of circles, so the outward unit normal can be computed analytically. For all numerical experiments, the boundary points are uniformly distributed with $w_i = \frac{1}{|\partial \Omega|}$, which corresponds to the periodic trapezoidal rule \cite{trefethen2014exponentially}. Note that \eqref{prop2} is satisfied under this construction. 

\subsection{Algebraic construction of the volume SBP operator} \label{subsec: constructing S}

\noindent After constructing $\fnt{E}_x$, the next step in constructing $\fnt{Q}_x$ is to determine a matrix $\fnt{S}_x$ such that $\fnt{Q}_x = \fnt{S}_x + \frac{1}{2}\fnt{E}_x$ satisfies \eqref{prop3}. Since the sparsity pattern of $\fnt{S}_x$ is not specified, we will determine a non-zero sparsity pattern for $\fnt{S}_x$ by building a connectivity graph between nodes $v_i \in S$. This will allow us to define an adjacency matrix $\fnt{A}$ on our set of points $S$. To determine the non-zero entries of $\fnt{S}_x$, we will utilize the approach taken to construct sparse low order multi-dimensional SBP operators in \cite{Chan_paper_1, wu2024entropy}, which adapts techniques from \cite{Inverting_L} involving the graph Laplacian of $\fnt{A}$. 
    \begin{definition}
        Provided a simple graph with the nodes $v_1, ... , v_N$, its corresponding adjacency matrix, $\fnt{A}$, is defined by the following:

        \begin{align}
           \fnt{A}_{ij}  = 
                \left\{
                    \begin{array}{lr}
                    1 & \text{if $i \neq j$ and $v_i$ is adjacent to $v_j$} \\
                    0 & \text{otherwise} 
                    \end{array}
                \right\}.
        \end{align}
        \label{def:adj}
    \end{definition}

    \noindent We now begin the second step, which is to construct an SBP operator that satisfies the consistency and SBP properties \eqref{prop1} and \eqref{prop3}. For the remainder of this section, we drop the $x$ subscript for simplicity of notation.\\
    
    \noindent From \eqref{prop3}, since $\fnt{E}$ is assumed to be known, we can construct $\fnt{Q} = \fnt{S} + \frac{1}{2}\fnt{E}$ if there exists a skew symmetric matrix $\fnt{S}$ such that \eqref{prop1} holds:
    \begin{gather}
         \fnt{Q1} = \fnt{0} \; \\
         \Rightarrow \fnt{S}\fnt{1} = -\frac{1}{2}\fnt{E1} = b\\
    \end{gather}
    where have introduced $\fnt{b} = -\frac{1}{2}\fnt{E1}$. Since $\fnt{S}$ is skew-symmetric by definiton, we follow \cite{Inverting_L} and make the ansatz that, for some $\fnt{\Psi} \in \mathbb{R}^N$
    \begin{equation}
        \fnt{S}_{ij} = \fnt{\Psi}_i - \fnt{\Psi}_j. \label{def of S}
    \end{equation}
    Hence:
    \begin{equation}
        (\fnt{S1})_i = \sum_j \fnt{S}_{ij} =  \sum_j \fnt{\Psi}_i - \fnt{\Psi}_j = \fnt{b}_i \label{graph laplacian eq 1}
    \end{equation}

    \noindent It was observed in \cite{Inverting_L} that \eqref{graph laplacian eq 1} is related to the graph Laplacian matrix $\fnt{L}$.
    \begin{definition}
         Given an adjacency matrix $A$, the corresponding Laplacian matrix $\fnt{L}$, is defined:
        \begin{align}
           \fnt{L}_{ij}  = 
                \left\{
                    \begin{array}{lr}
                    \text{deg($v_i) = \sum_{j = 1}^N \fnt{A}_{ij}$} & \text{if $i = j$} \\
                    -1 & \text{if $\fnt{A}_{ij} \neq 0$} \\
                    0 & \text{otherwise}
                    \end{array}
                \right\}.
        \end{align}
    \end{definition}
    \noindent Then, \eqref{graph laplacian eq 1}, is equivalent to the following property of the graph Laplacian, which holds for an arbitrary vector $\fnt{x}$:
    \begin{equation}
        (\fnt{Lx})_i  = \sum_{v_j \in \text{Nbrs}(v_i)} (x_i - x_j) \label{graph laplacian eq 2}
    \end{equation}

    \noindent Notice that \eqref{graph laplacian eq 1} is in the same form as \eqref{graph laplacian eq 2}, allowing us to establish the following relationship:
    \begin{gather}
        (\fnt{S1})_i = \fnt{b}_i = \sum_j \fnt{\Psi}_i - \fnt{\Psi}_j = (L\fnt{\Psi})_i  \label{graph laplacian eq 3} \\
        \fnt{S1} = \fnt{b} \Leftrightarrow \fnt{L\fnt{\Psi}} = \fnt{b}. \label{systems of equations}
    \end{gather}

    \noindent We will later discuss different methods for constructing the adjacency matrix $\fnt{A}$, but assuming we have defined a notion of connectivity such that the graph formed by the nodes is connected, we still need to ensure that for any $\fnt{b} = -\frac{1}{2}\fnt{E1}$ that $\fnt{L\fnt{\Psi}} = \fnt{b}$ does indeed have a solution since $\fnt{L}$ is singular. \\
    \begin{lemma}
    For $\fnt{b} = -\frac{1}{2}\fnt{E}_x\fnt{1}$, where $\fnt{E}_x$ satisfies \eqref{prop3b},  $\fnt{L\fnt{\Psi}} = \fnt{b}$ has a solution. 
    \end{lemma}
    \begin{proof} 
    \begin{equation}
        \fnt{b} \in R(L) \Leftrightarrow \fnt{b} \perp N(L^T) \Leftrightarrow \fnt{b} \perp N(L)
    \end{equation}
    \noindent However, because we assumed that L was a graph Laplacian to a connected graph. $\fnt{L}$ is a positive semi definite matrix with only one zero eigenvalue. Hence its null space has a dimension of 1. In addition, by definition of a graph Laplacian: $\fnt{L1} = \fnt{0}$. Therefore,
    \begin{equation}
        N(\fnt{L}) = \text{span}\{\textbf{1}\}.
    \end{equation}
    Hence $\fnt{L\fnt{\Psi}} = \fnt{b}$ has a solution if and only if $\fnt{b} \perp \text{span}\{\textbf{1}\}$. 
    By definition of $\fnt{E}_x$ \eqref{prop3b}:
    \begin{equation}
        \textbf{1}^Tb = 0 \Leftrightarrow \textbf{1}^T\fnt{E}_x\textbf{1} = \int_{\partial \Omega} n_x = 0,
        \end{equation}
    implying that $\fnt{b} \perp \text{span}\{\textbf{1}\}$.
    \end{proof}
    
    \noindent In practice, because $\fnt{L}$ has a null space with dimension of 1, there are infinite solutions for $\fnt{L\fnt{\Psi}} = \fnt{b}$. Therefore, an extra linearly independent constraint is added (ie: $\textbf{1}^T \fnt{\fnt{\Psi}} = 0$) to make the solution of the graph Laplacian problem unique. \\
    
    \noindent From $\fnt{\Psi}$, we can compute $\fnt{S}$ \eqref{def of S}. Since we assume we are given $\fnt{E}$, we now can construct $\fnt{Q}$ from $\fnt{S}, \fnt{E}$. The last step is to construct a suitable diagonal norm matrix $\fnt{H}$. \\    

    \subsection{Optimization of the norm matrix} \label{subsec: Creating H}

    \noindent In Section~\ref{subsec: constructing E} and Section~\ref{subsec: constructing S}, we detail a method for constructing SBP operators $\fnt{Q}$. In order to construct differentiation operators $\fnt{D} = \fnt{H}^{-1}\fnt{Q}$ which will be used to discretize a system of PDEs, what remains is to construct the norm matrix $\fnt{H}$. \\

    \noindent From the conditions imposed on $\fnt{Q}_x$ and $\fnt{Q}_y$, $\fnt{D}_x\textbf{1} = \fnt{0}$ and $\fnt{D}_y\textbf{1} = \fnt{0}$ by construction. These are first order consistency conditions. However, while $\fnt{D}_x$ is supposed to be an approximation to the first derivative, $\fnt{D}_x \fnt{x} \neq \fnt{1}$. Thus, we choose to optimize the accuracy of both $\fnt{D}_x, \fnt{D}_y$ by constructing a diagonal norm matrix $\fnt{H}$ to minimize the error in $\fnt{D}_x \fnt{x} \approx \fnt{1}$ and $\fnt{D}_y \fnt{y} \approx \fnt{1}$. \\
    
     \noindent We construct the norm matrix as follows: let $\fnt{w}$ be the diagonal of $\fnt{H}$. Then, instead of directly minimizing the difference between $\fnt{D}_x \fnt{x} - \fnt{1}$, we can multiply through by $\fnt{H}$. Noting that $\fnt{H}\fnt{1} = \fnt{w}$, we can then minimize the difference between $\fnt{Q}_x\fnt{x}-\fnt{w}$ and $\fnt{Q}_y\fnt{y}-\fnt{w}$ instead. This translates to solving the following non-negative least squares problem: 
    \begin{equation}
        \text{min}_{\fnt{w} > \fnt{0}} \frac{1}{2}\|\fnt{Q}_x\fnt{x} - \fnt{w}\|^2 +  \frac{1}{2}\|\fnt{Q}_y\fnt{y} - \fnt{w}\|^2 \label{quadratic op}
    \end{equation}
    In practice, instead of enforcing a strict inequality, we enforce $\fnt{w} > 1 / N^2$ where $N$ is the total number of points. Furthermore, because solving \eqref{quadratic op} for large numbers of points can be computationally challenging, we utilize a splitting conic solver \cite{scs}, which splits the solution of \eqref{quadratic op} into an iterative process involving the solution of linear systems and a projection onto the space of positive weights (e.g., a cutoff).
    \\
    
    \subsection{Accuracy of the optimized norm matrix} \label{subsec: testing H}

      \noindent In this section, we numerically compare the accuracy of $\fnt{D}_x$ under the optimized norm matrix $\fnt{H}_{\rm opt}$ (determined by solving \eqref{quadratic op}) and under a simple uniformly weighted norm matrix, $\fnt{H}_{\rm unif}$, where $(\fnt{H}_{\rm unif})_{ii}  = \frac{\text{Vol}(\Omega)}{|S|}$. Note that $(\fnt{H}_{\rm unif})_{ii} = \frac{\text{Vol}(\Omega)}{|S|}$ satisfies $\fnt{1}^T\fnt{H}_{\rm unif}\fnt{1}  = \text{Vol}(\Omega)$, implying that \eqref{int1} is exact for $u,v = 1$. \\
      
      \noindent To compare $\fnt{H}_{\rm opt}$ and $\fnt{H}_{\rm unif}$, we test the accuracy of the differentiation matrices under each norm matrix by approximating the derivative of two functions $\frac{\partial u_1}{\partial x} $, $\frac{\partial u_2}{\partial x}$ on $S$, where
    \begin{enumerate}
        \item $u_1(x,y) = x$
        \item $u_2(x,y) =  4\sin(y)e^{-(x^2 + y^2)}(\cos(x) - 2x\sin(x))$
    \end{enumerate}
    We define $\Omega_1 := \{(x, y) \in \mathbb{R}^2 \;|\; x^2 + y^2 \le R^2 \}$ and $S$ as $S = S_i \cup S_{b}$ where
    \begin{gather*}
        S_i := \{(x, y) \in \Omega_1 \; |\; (x,y) = R\left(\frac{2}{n_x}a - 1, \frac{2}{n_y}b - 1 \right), \quad a \in \mathbb{W}, \quad b \in \mathbb{W}\} \\
         S_b := \{(x, y) \in \partial \Omega_1 \;|\; (x,y) = \left(R\cos\left(\frac{2 \pi}{n_b}t\right), R\sin\left(\frac{2 \pi}{n_b}t\right)\right), \quad t \in \mathbb{W}, \quad t \le n_b\} 
    \end{gather*} 
    where $\mathbb{W} = \{0, 1, 2, 3, \ldots\}$. In other words, the set $S$ is the union of the interior and boundary points. Here, $n_x$ and $n_y$ respectively refer to the number of evenly spaced points along the $x$ and $y$ coordinates used to define a ``background'' grid. The circle defined by $\Omega_1$, and $n_b$ refers to the number of evenly spaced points on the boundary of the circle starting at $(R, 0)$ and going counter-clockwise. Figure \ref{Simple mesh} illustrates this construction of $S$ for $R = 3$, which we will also use for following numerical experiments.

    \begin{figure}[h]
        \centering
        \includegraphics[scale = 0.1]{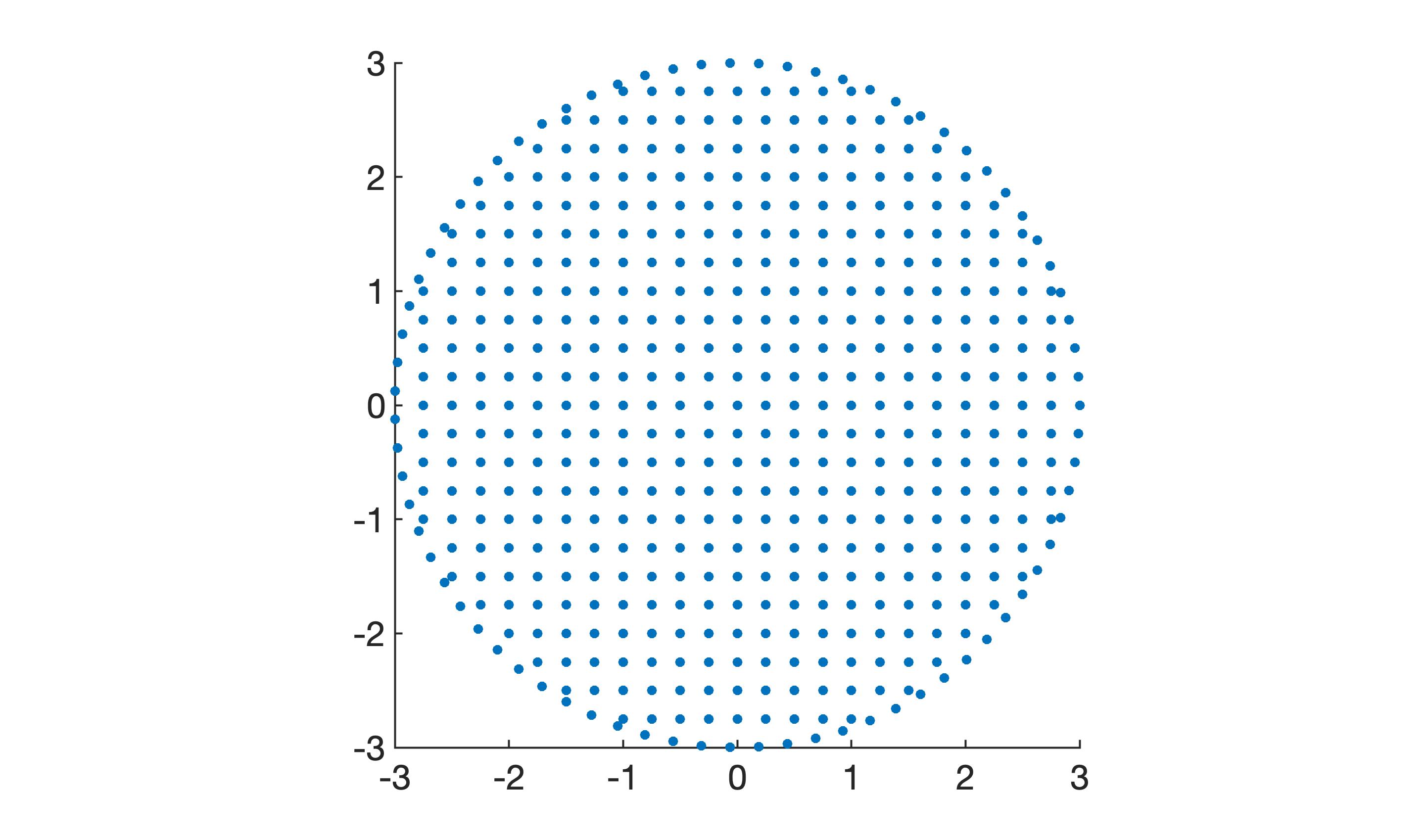}
        \caption{The domain $\Omega_1$ for $n_x = n_y = 25, n_b = 75, R = 3$}
        \label{Simple mesh}
    \end{figure}

    \noindent We approximate the $x$ derivative at nodal points via $\fnt{H}^{-1}\fnt{Q}_x \fnt{u} \approx \frac{\partial u}{\partial x}$, where $\fnt{H}$ can be either $\fnt{H}_{\rm opt}$ or $\fnt{H}_{\rm unif}$. We follow an analogous procedure for the $y$ derivative. The $L_2$ error of the partial derivative of $u(x, y)$ with respect to $x$ can be approximated by:

    \begin{equation}
        \sqrt{(\fnt{D}_x \fnt{u} - \fnt{u}_x)^T\fnt{H}(\fnt{D}_x \fnt{u} - \fnt{u}_x)},    
    \end{equation} 
    where $\fnt{u}_x$ is the vector containing point values of the exact derivative of $u$.\\ 
   
    \noindent Tables \ref{L2 error optimized H} and \ref{L2 error uniform H} show the $L_2$ errors when computing $\frac{\partial u_1(x, y)}{\partial x}$ and $\frac{\partial u_2}{\partial x}$ using the SBP operators created under the two different norm matrices $\fnt{H}_{\rm opt}, \fnt{H}_{\rm unif}$ on $\Omega_1$. 
    The adjacency matrix $\fnt{A}$ (see Definition~\ref{def:adj}) is computed as follows: a point $(x_j, y_j)$ is a neighbor of a point $(x_i, y_i)$ if the distance between them is smaller than some distance. In these experiments, we use an arbitrary distance threshold of
    \[
    r = 2.5\frac{{\rm Diam}(\Omega_1)}{\max(n_x, n_y)},
    \]
    where ${\rm Diam}(\Omega_1)$ is the diameter of the circular domain $\Omega_1$. We compare this method (which we refer to as the ``Euclidean Radius'' methods) with other methods of computing the adjacency matrix in Section \ref{sec: finding adjacency}.\\
    
    \noindent Table~\ref{L2 error optimized H} and \ref{L2 error uniform H} also show the computed convergence rates: $\text{log}_2\frac{error_{i - 1}}{error_i}$. The operators are tested for the following grid sizes:     
    \begin{itemize}
        \setlength{\itemsep}{-0.05em}
        \item Grid 1: $n_x = n_y = 75,  n_b = 250$
        \item Grid 2: $n_x = n_y = 150,  n_b = 500$
        \item Grid 3: $n_x = n_y = 300,  n_b = 1000$
        \item Grid 4: $n_x = n_y = 600,  n_b = 2000$
        \item Grid 5: $n_x = n_y = 1200,  n_b = 4000$
    \end{itemize}

    \captionsetup{font=small}  

    \newcolumntype{P}[1]{>{\centering\arraybackslash}p{#1}}

\begin{table}[h]
    \centering
    \renewcommand{\arraystretch}{1.5} 
    \begin{tabular}{|P{1.3cm}|c|c|c|c|c|}
    \hline
    \multicolumn{1}{|c|}{Grid} & \multicolumn{1}{c|}{$\frac{\partial u_1 }{\partial x}$} & \multicolumn{1}{c|}{Convergence Rate} & \multicolumn{1}{c|}{$\frac{\partial u_2}{\partial x}$} & \multicolumn{1}{c|}{Convergence Rate}  \\
    \hline
    $1$ & $0.2605$ &  & $0.06413$ &  \\
    \hline
    $2$ & $0.2013$ & $0.3719$ & $0.01626$ & $1.980$  \\
    \hline
    $3$ & $0.1439$ & $0.4842$ & $0.004061$ & $2.001$ \\
    \hline
    $4$ & $0.1072$ &$0.4248$ & $0.001205$ & $1.753$ \\
     \hline
    $5$ & $0.07634$ & $0.4896$ & $0.0002794$ & $2.109$  \\
     \hline
    \end{tabular}
    \caption{$L_2$ error for $\frac{\partial u_1 }{\partial x} | \frac{\partial u_2}{\partial x} $ using $\fnt{H}_{\rm opt}$ on $\Omega_1$.}
    \label{L2 error optimized H}
\end{table}
    
\begin{table}[h]
    \centering
    \renewcommand{\arraystretch}{1.5} 
    \begin{tabular}{|P{1.3cm}|c|c|c|c|c|}
    \hline
    \multicolumn{1}{|c|}{Grid} & \multicolumn{1}{c|}{$\frac{\partial u_1 }{\partial x}$} & \multicolumn{1}{c|}{Convergence Rate} & \multicolumn{1}{c|}{$\frac{\partial u_2}{\partial x}$} & \multicolumn{1}{c|}{Convergence Rate}  \\
    \hline
    $1$ & $0.6704$ &  & $0.02969$ &  \\
    \hline
    $2$ & $0.4763$ & $0.4932$ & $0.02008$ & $0.5642$  \\
    \hline
    $3$ & $0.3328$ & $0.5712$ & $0.01286$ & $0.6429$ \\
    \hline
    $4$ & $0.2389$ &$0.4782$ & $0.007390$ & $0.7992$ \\
     \hline
    $5$ & $0.1690$ & $0.4994$ & $0.003894$ & $0.9243$  \\
     \hline
    \end{tabular}
     \caption{$L_2$ error for $\frac{\partial u_1 }{\partial x} | \frac{\partial u_2}{\partial x} $ using $\fnt{H}_{\rm unif}$ on $\Omega_1$.}
    \label{L2 error uniform H}
\end{table}

    \noindent From Tables \ref{L2 error optimized H} and \ref{L2 error uniform H}, we observe that the errors are consistently larger when using $\fnt{H}_{\rm unif}$. This suggests that using an optimized norm matrix $\fnt{H}_{\rm opt}$ as discussed in Section~\ref{subsec: Creating H} reduces the approximation error. 
    
\subsection{On the observed convergence rates}
\label{sec:convergence_rates}
\noindent We notice that in Tables \ref{L2 error optimized H} and \ref{L2 error uniform H} that the convergence rates are significantly lower for $\frac{\partial u_2}{\partial x}$ when compared to $\frac{\partial u_1 }{\partial x}$. In particular, when using $\fnt{H}_{\rm opt}$, the convergence rate for $\frac{\partial u_1 }{\partial x}$ is around $\frac{1}{2}$ while it is around 2 for  $\frac{\partial u_2 }{\partial x}$. This section will present numerical experiments which suggest that this difference in observed convergence rates is due to accuracy of the interior vs boundary stencils. \\

\noindent As we can see from Figure \ref{fig:abs val errors simple}, the approximation error $\left|\fnt{D}_x\fnt{u} - \fnt{u}_x\right|$ is larger near the boundaries. However, since $\frac{\partial u_2}{\partial x}$ is close to $0$ at the boundary $\partial \Omega_1$, the errors near the boundaries are smaller for $u_2$. This is not the case for $u_1$. \\

\begin{figure}[H]
    \centering
    
    
    \begin{subfigure}[b]{0.49\textwidth}
        \centering
        \includegraphics[width=\textwidth]{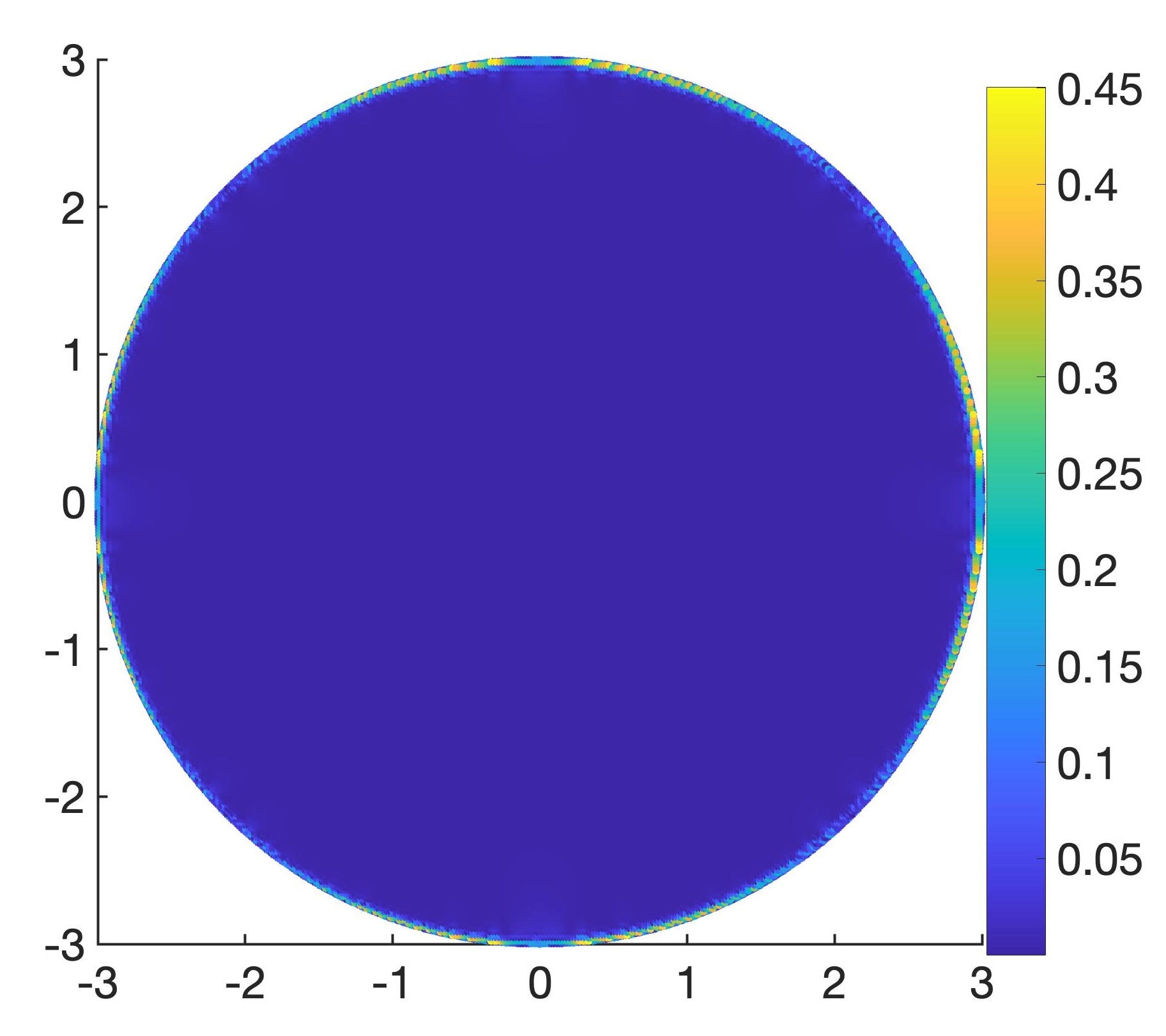}
        \caption{$|\fnt{D}_x\fnt{u_1} - \fnt{u}_{\fnt{1}x}|$}
        \label{fig: u_1_diff}
    \end{subfigure}
    \begin{subfigure}[b]{0.49\textwidth}
        \centering
        \includegraphics[width=\textwidth]{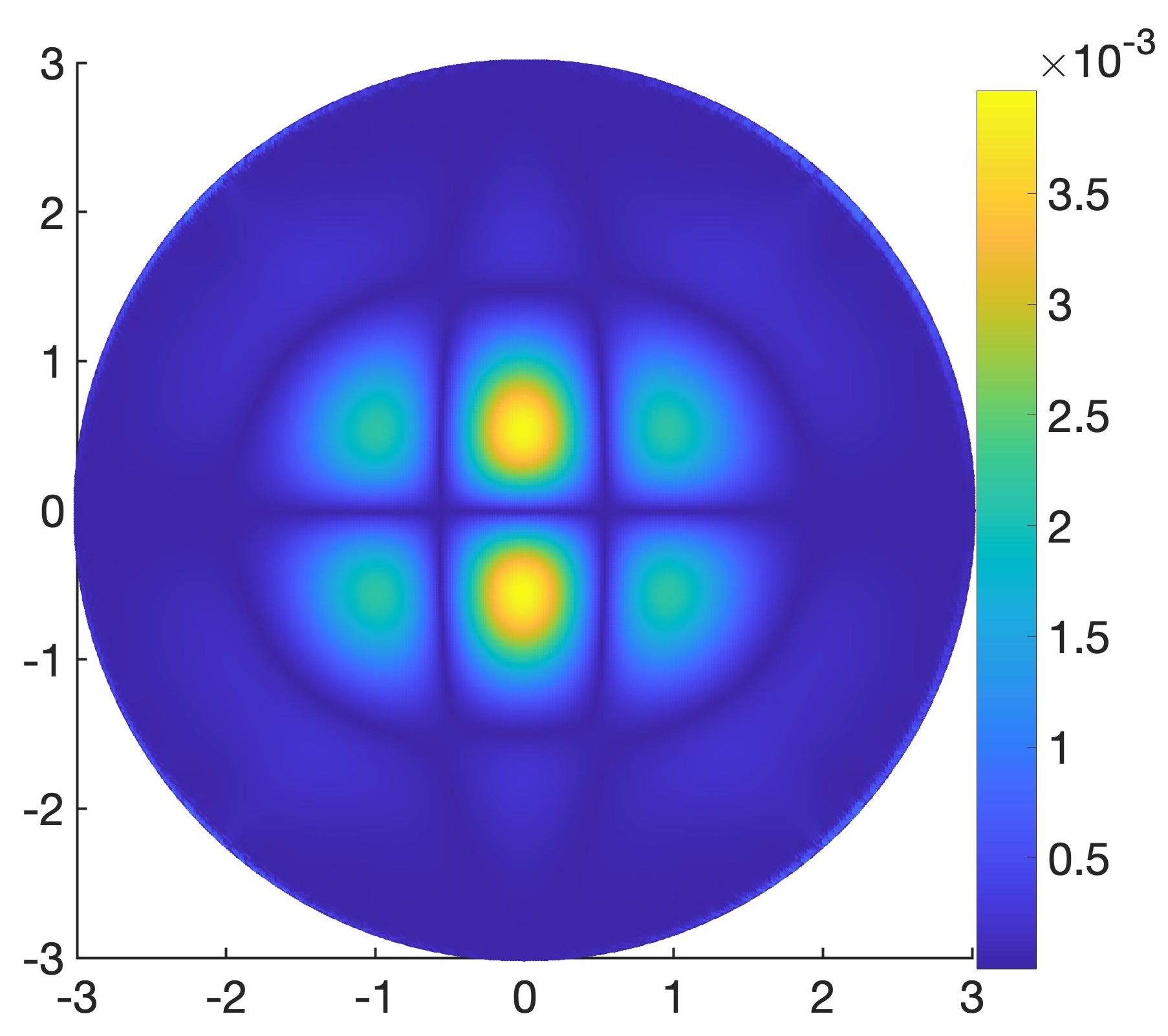}
        \caption{$|\fnt{D}_x\fnt{u_2} - \fnt{u}_{\fnt{2}x}|$}
        \label{fig:u_2_diff}
    \end{subfigure}
    \caption{Absolute values of the computed errors for the differential operators on $u_1$ and $u_2$ on Grid 3 using the ``Euclidean Radius'' adjacency method.}
    \label{fig:abs val errors simple}
\end{figure}

\noindent To further support the analysis from above, we perform the same numerical experiment of approximating $\frac{\partial u_1}{\partial x} $ and $\frac{\partial u_2}{\partial x}$, but this time on $\Omega_2$, which is a circle with three smaller circles cut out of the interior:
\begin{gather*}
    C_i := \{(x, y) \in \mathbb{R}^2 \;|\; (x-h_i)^2 + (y-k_i)^2 \le s_i \} \\
    \Omega_2 := \{(x, y) \in \mathbb{R}^2 \;|\; x^2 + y^2 \le R^2, \;\; (x, y) \notin C_i, \; \; \forall i \in \{1, 2, .. , m\}\}.
\end{gather*}

\noindent We set $R = 3$, $s_i = \frac{2}{3}$, $h_i = \frac{3}{2}\cos\left(\frac{2\pi}{3}i\right)$, and $k_i = \frac{3}{2}\sin\left(\frac{2\pi}{3}i\right)$, for $i \in \{1, 2, 3\}$. By picking such a $\Omega_2$, $\frac{\partial u_2}{\partial x}$ is not approximately zero near the inner boundaries (ie: $\partial C_i$). Figure \ref{Sharingan mesh} shows $\Omega_2$ as reference. \\
\begin{figure}[h]
            \centering
            \includegraphics[scale = 0.1]{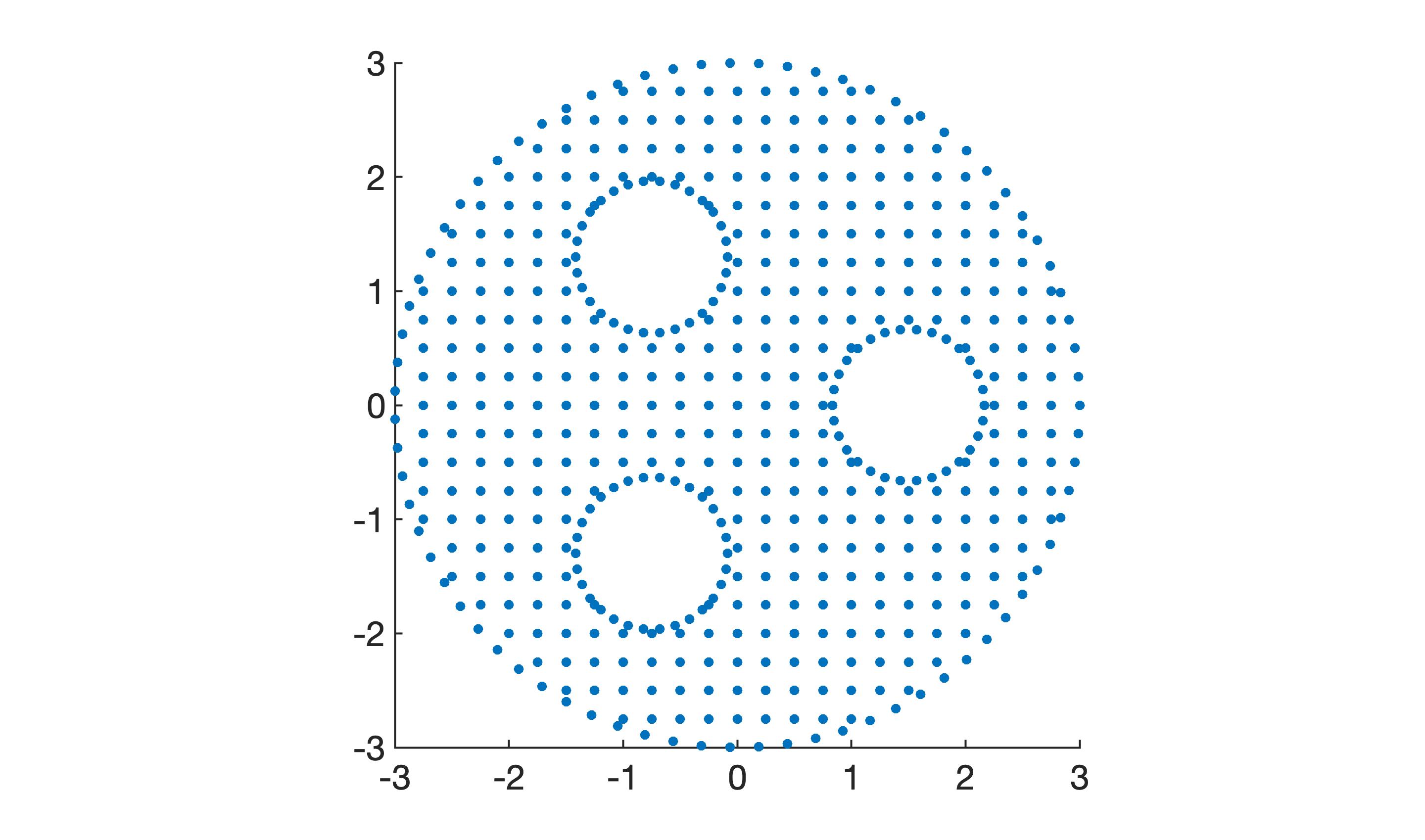}
            \captionsetup{belowskip=0pt}
            \caption{The domain $\Omega_2$ for $n_x = n_y = 25, n_b = 75, n_i  = 30, R = 3$. Here, $n_i$ denotes the number of nodes on the boundary of each inner circle.}
            \label{Sharingan mesh}
\end{figure}

\noindent  Table \ref{L2 sharingan error} displays the $L_2$ errors of $\frac{\partial u}{\partial x}$ on $\Omega_2$. 
Because $\frac{\partial u_2}{\partial x}$ is not near 0 near the inner boundaries of $\Omega_2$, we expect the convergence rates for both $\frac{\partial u_2}{\partial x}$ and $\frac{\partial u_1(x, y)}{\partial x}$ to both be around $\frac{1}{2}$. The grid sizes for $\Omega_2$ are as follows:
\begin{itemize}
        \setlength{\itemsep}{-0.05em}
        \item Grid 1: $n_x = n_y = 75,  n_b = 250, n_i = 60$
        \item Grid 2: $n_x = n_y = 150,  n_b = 500, n_i = 120$
        \item Grid 3: $n_x = n_y = 300,  n_b = 1000, n_i = 240$
        \item Grid 4: $n_x = n_y = 600,  n_b = 2000, n_i = 480$
        \item Grid 5: $n_x = n_y = 1200,  n_b = 4000, n_i = 960$
\end{itemize}
where $n_i$ denotes the number of equispaced quadrature points placed on the boundary of $C_i$.\\

\noindent Table \ref{L2 sharingan error} show that the convergence rates are now similar for both $\frac{\partial u_1(x, y)}{\partial x}$ and $\frac{\partial u_2}{\partial x}$ at a rate of $\frac{1}{2}$. This is consistent with our expectation that the convergence rates should be similar for functions in which their values and their derivatives are not near 0 at the boundaries. From Figure \ref{abs val errors sharingan}, we see that the error is again greatest near the boundaries. \\

\noindent We further verify the observation made above about the worse behavior near the boundaries. We check the $L_\infty$ error of $\frac{\partial u_1 }{\partial x}$ on $\Omega_1$ but this time differentiate between the nodes near the boundary: $\{(x, y) \in S | x^2 + y^2 \ge 4 \}$ and nodes far from the boundary $\{(x, y) \in S | x^2 + y^2 < 4 \}$. We use the $L_\infty$ error instead of the $L_2$ error to avoid needing to define appropriate quadratures over subsets of points. Table \ref{L2 error simple interior exterior} shows the results and convergence rates.

\begin{table}[h]
    \centering
    \renewcommand{\arraystretch}{1.5} 
    \begin{tabular}{|P{1.3cm}|c|c|c|c|c|}
    \hline
    \multicolumn{1}{|c|}{Grid} & \multicolumn{1}{c|}{$\frac{\partial u_1 }{\partial x}$} & \multicolumn{1}{c|}{Convergence Rate} & \multicolumn{1}{c|}{$\frac{\partial u_2}{\partial x}$} & \multicolumn{1}{c|}{Convergence Rate}  \\
    \hline
    $1$ & $0.3547$ &  & $0.1622$ & \\
    \hline
    $2$ & $0.2695$ & $0.3963$ & $0.1000$ & $0.6978$ \\
    \hline
    $3$ & $0.1973$ & $0.4499$ & $0.07470$ & $0.4208$ \\
    \hline
    $4$ & $0.1440$ & $0.4543$ & $0.05211$ & $0.5195$ \\
     \hline
    $5$ & $0.1018$ & $0.5003$ & $0.03841$ & $0.4400$    \\
     \hline
    \end{tabular}
    \caption{$L_2$ errors and computed convergence rates for $\frac{\partial u_1 }{\partial x}$ and $\frac{\partial u_2}{\partial x}$  on $\Omega_2$.}       
    \label{L2 sharingan error}
\end{table}

\begin{table}[h]
    \centering
    \renewcommand{\arraystretch}{1.5} 
    \begin{tabular}{|P{1.3cm}|c|c|c|c|c|}
    \hline
    \multicolumn{1}{|c|}{Grid} & \multicolumn{1}{c|}{$x^2 + y^2 > 4$} & \multicolumn{1}{c|}{Convergence Rate} & \multicolumn{1}{c|}{$x^2 + y^2 \le 4$} & \multicolumn{1}{c|}{Convergence Rate}  \\
    \hline
    $1$ & $0.4213$ &  & $0.005472$ &  \\
    \hline
    $2$ & $0.4819$ & $-0.1938$ & $0.002672$ & $1.0342$  \\
    \hline
    $3$ & $0.4505$ & $0.09721$ & $7.631 \times 10^{-4}$ & $1.8081$ \\
    \hline
     $4$ & $0.4767$ &$-0.08155$ & $3.250 \times 10^{-4}$ & $1.2312$ \\
     \hline
     $5$ & $0.5122$ & $-0.1036$ & $9.399 \times 10^{-5}$ & $1.7901$  \\
     \hline
    \end{tabular}
     \caption{$L_\infty$ error for $\frac{\partial u_1 }{\partial x} $ at interior vs exterior nodes.}    
    \label{L2 error simple interior exterior}
\end{table}

\noindent From Table \ref{L2 error simple interior exterior}, we do indeed see that the error is much more significant (a few order of magnitudes) larger near the boundaries. Furthermore, combined with our observations from Table \ref{L2 error optimized H} and \ref{L2 sharingan error}, it also suggests (though would need to be proven) that for sufficiently close nodes to the boundaries, the convergence rate is 0.5 and the convergence rate for nodes sufficiently far from the boundary nodes achieve between first and second order accuracy. This suggests that the meshfree interior stencils are second order accurate, while the boundary stencils are inconsistent in the sense that the pointwise error does not converge to zero as the mesh size decreases. However, experiments in Section~\ref{sec: results} suggest that numerical approximations of solutions of PDEs achieve first order accuracy in the $L^2$ norm. It is possible that the meshfree operators constructed here are first order accurate in some sense, but not in a pointwise sense. Analyzing this will be the focus of future work. 

\begin{figure}
    \centering
    
    
    
    \begin{subfigure}[b]{0.49\textwidth}
        \centering
        \includegraphics[width = 7.5cm, height = 7 cm]{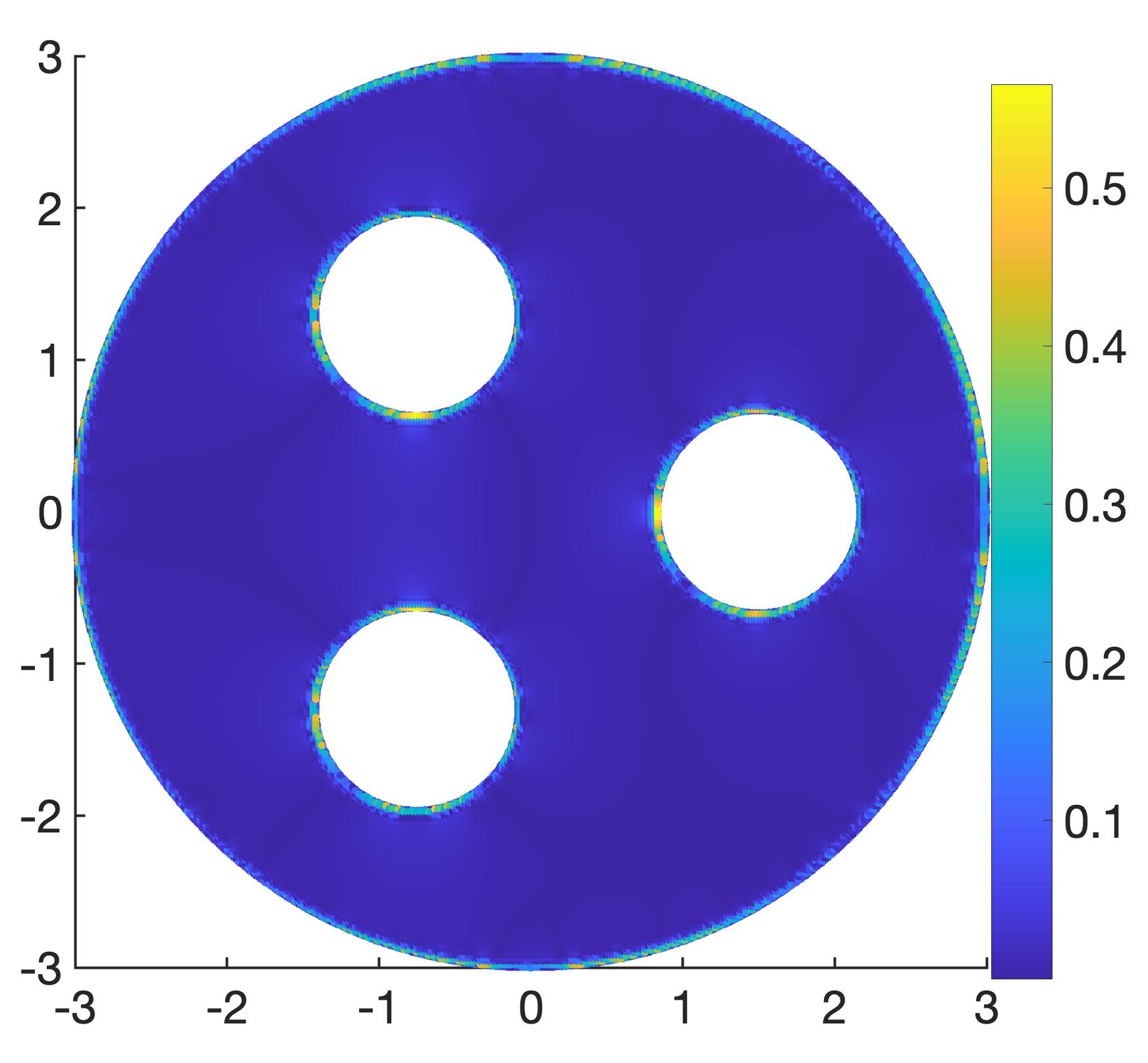}
        \caption{$|\fnt{D}_x\fnt{u_1} - \fnt{u}_{\fnt{1}x}|$}
        \label{fig: u_1_sharingan_diff}
    \end{subfigure}
    \hspace{0.001\textwidth}
    \begin{subfigure}[b]{0.49\textwidth}
        \centering
        \includegraphics[width= 7.5cm, height = 7cm]{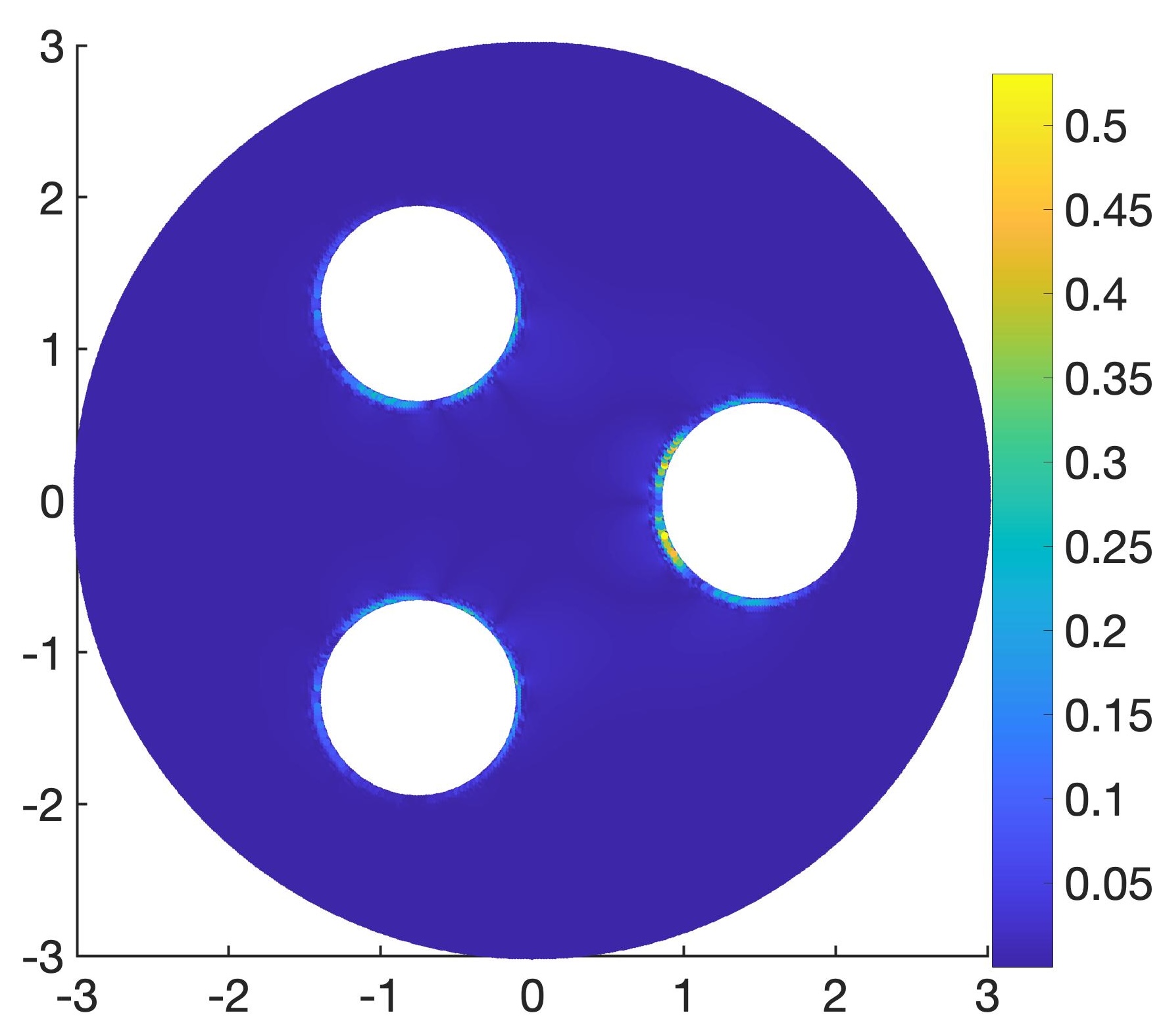}
        \caption{$|\fnt{D}_x\fnt{u_2} - \fnt{u}_{\fnt{2}x}|$}
        \label{fig:u_2_sharingan_diff}
    \end{subfigure}

    \caption{Absolute value errors for differential operators on $u_1$ and $u_2$ on $\Omega_2$ (Euclidean Radius, Grid 3).}
     \label{abs val errors sharingan}
\end{figure}

\section{Finding Suitable Notion of Adjacency} \label{sec: finding adjacency}
    \noindent An important detail of the construction that was skipped over in previous sections is the construction of the adjacency/connectivity matrix $\fnt{A}$. In this section, we will explore several notions of connectivity through numerical experiments and determine which produces meshfree operators which result in lower errors. We will also discuss a few implementation details. While the only requirement on the adjacency matrix $\fnt{A}$ is that it represents a fully connected graph between the nodes, it is important to make sure that the notion of connectivity defined maintains sparsity but does not decrease accuracy. \\
    
\noindent  Let $v_i$ be the node corresponding to $(x_i, y_i) \in S$ and let $d(v_i, v_j)$ correspond to the Euclidean distance between $v_i$ and $v_j$. We explore the following four notions of connectivity: 

    \begin{enumerate}[topsep=0.5pt]
        \item \textbf{Euclidean Radius}: we say that $v_i$ and $v_j$ are adjacent if $d(v_i, v_j) \le r$.  Section \ref{Notes on Implementation} gives further details on the efficient implementation. 
        \item \textbf{Minimum spanning tree}: Using the nodes $v_1, ..., v_N$ and $d(v_i, v_j)$, we can form the adjacency matrix from a minimum spanning tree (the graph with the minimum number of edges such that there is a path between every pair of nodes). 
        \item \textbf{Degree 1 Delaunay}: We call $v_i$ and $v_j$ adjacent if they are neighbors in a Delaunay triangulation on $S$.  
        \item \textbf{Degree 2 Delaunay}: We call $v_i$ and $v_j$ adjacent if they are neighbors in a Delaunay triangulation on $S$ or if  $v_i$ and $v_j$ share a common neighbor. 
    \end{enumerate}

    \begin{remark}
    \noindent We note that the Delaunay triangulation connectivity methods are not ``meshfree''. However, unlike mesh-based finite element methods, the quality and regularity of the mesh generated by Delaunay triangulation does not significantly impact solution quality \cite{FEMgeometry}. Moreover, we consider Delaunay triangulations only as a comparison against truly meshfree notions of adjacency, such as the Euclidean Radius.
    \end{remark}
    
\begin{table}[H]
    \centering
    \renewcommand{\arraystretch}{1.5} 
    \begin{tabular}{|P{1.6cm}|p{1.4cm}|p{1.1cm}|p{1.2cm}|p{1.0cm}|p{1.2cm}|p{1.0cm}|}
    \hline
    \multicolumn{1}{|c|}{Grid} & \multicolumn{2}{c|}{Euclidean Radius} & \multicolumn{2}{c|}{Delaunay Deg 1} & \multicolumn{2}{c|}{Delaunay Deg 2}  \\
    \hline
     & Error & Rate  & Error & Rate & Error & Rate \\
    \hline
    $1$ & $0.2605$ &  & $0.2909$ & & $0.2103$ & \\
    \hline
    $2$ & $0.2013$ & $0.3719$ & $0.2246$ & $0.3732$ & $0.1656$ & $0.3447$  \\
    \hline
    $3$ & $0.1439$ &  $0.4248$ & $0.1573$ & $0.5318$  & $0.1134$ & $0.5463$  \\
    \hline
     $4$ & $0.1072$ & $0.4248$& $0.1163$ & $0.4357$ & $0.0853 $ & $0.4108$   \\
     \hline
     $5$ & $0.07635$ & $0.4896$ & $0.08369$ & $0.4747$  & $0.0607$ & $0.4908$   \\
     \hline
    \end{tabular}
    \caption{$L_2$ error for $\frac{\partial u_1 }{\partial x}$ and computed convergence rates on $\Omega_1$.}   
    \label{L2 error simple_1}
\end{table}

\begin{table}[H]
    \centering
    \renewcommand{\arraystretch}{1.5} 
    \begin{tabular}{|P{1.6cm}|p{1.6cm}|p{1.0cm}|p{1.4cm}|p{1.0cm}|p{1.4cm}|p{1.0cm}|}
    \hline
    \multicolumn{1}{|c|}{Grid} & \multicolumn{2}{c|}{Euclidean Radius} & \multicolumn{2}{c|}{Delaunay Deg 1} & \multicolumn{2}{c|}{Delaunay Deg 2}  \\
    \hline
     & Error & Rate & Error & Rate & Error & Rate \\
    \hline
     $1$ & $0.06413$ &  & $0.03321$ &  & $0.0700$ &  \\
    \hline
    $2$ & $0.01626$ & $1.980$ & $0.01591$ & $1.062$ & $0.0197$ & $1.829$  \\
    \hline
    $3$ & $0.004061$ & $2.001$ & $0.008208$ & $0.9548$ & $0.00753$ &  $1.387$  \\
    \hline
    $4$ & $0.001205$ & $1.753$ &$0.004118$ & $0.9951$ & $0.003561$ &  $1.080$ \\
    \hline
    $5$& $0.0002794$ & $2.109$  & $0.002072$ & $0.9909$  & $0.001787$ &  $0.9947$ \\
    \hline
    \end{tabular}
    \caption{$L_2$ error for $\frac{\partial u_2 }{\partial x}$ and computed convergence rates on $\Omega_1$.}      
    \label{L2 error simple_2}
\end{table} 
\noindent We test the accuracy of the differentiation matrices created by the four different notions of adjacency by approximating $\frac{\partial u_1}{\partial x} $, $\frac{\partial u_2}{\partial x}$. Tables \ref{L2 error simple_1} and \ref{L2 error simple_2} display the $\frac{\partial u_1}{\partial x}$, $\frac{\partial u_2}{\partial x}$ $L_2$ errors when approximated by the SBP operators with the various adjacency methods on $\Omega_1$ ($R = 3$). \\

\noindent All Euclidean Radius methods used $r = 2.5 \frac{6}{n_x}$, which we chose heuristically to balance accuracy and sparsity of the resulting differentiation operators. Additionally, while we do not show them, the $L_2$ errors for $\frac{\partial u_1 }{\partial x}$, $\frac{\partial u_1 }{\partial y}$, $\frac{\partial u_2}{\partial x}$, and $\frac{\partial u_2}{\partial y}$ are very similar to the errors for $\frac{\partial u_1 }{\partial x}$ and $\frac{\partial u_2}{\partial x}$, which are shown in Tables~\ref{L2 sharingan adjacency error 1} and \ref{L2 sharingan adjacency error 2}. \\

\noindent We make a couple of observations. The first observation is that determining the adjacency matrix using a minimum spanning tree fails to produce accurate meshfree operators. We do not display the results in Table~\ref{L2 error simple_1} and Table~\ref{L2 error simple_2} because the error blows up as the number of points increases. The second observation is that determining the adjacency matrix using the Delaunay degree 1 method results in a larger error than the determining the adjacency matrix using the Delaunay degree 2 method. Both Delaunay degree 2 and Euclidean radius methods result in similar errors for this test case and show similar convergence behavior.\\

\noindent When computing $\frac{\partial u_1 }{\partial x}$, the rate of convergence of the error is approximately $0.5$, while the rate of convergence of the error when computing $\frac{\partial u_2 }{\partial x}$ is around $2$ for the Euclidean Radius method and $1$ for the Delaunay degree 2 method. Note that, since $\frac{\partial u_2}{\partial x}\approx 0$ near the boundaries (as discussed in Section~\ref{sec:convergence_rates}), this effectively tests only accuracy of the interior stencil. These results suggest that interior stencils constructed using the Euclidean Radius approach are second order accurate, while interior stencils constructed using the Delaunay degree 2 approach are only first order accurate. \\

\noindent  To differentiate between the Euclidean Radius and Delaunay Degree 2 notion of adjacency, we perform the same numerical experiment but this time on $\Omega_2$. Tables \ref{L2 sharingan adjacency error 1} and \ref{L2 sharingan adjacency error 2} display the $L_2$ errors for the Euclidean Radius and Delaunay Degree 2 methods on $\Omega_2$. \\


\begin{table}
    \centering
    \renewcommand{\arraystretch}{1.5} 
    \begin{tabular}{|P{1.3cm}|c|c|c|c|c|}
    \hline
    \multicolumn{1}{|c|}{Grid} & \multicolumn{2}{c|}{Euclidean Radius}  & \multicolumn{2}{c|}{Delaunay Deg 2}  \\
    \hline
    $1$ & $0.3547$ &  & $0.3925,$ & \\
    \hline
    $2$ & $0.2695$ & $0.3963$ & $0.2757$ &  $0.5096$ \\
    \hline
    $3$ & $0.1973$ & $0.4499$ & $0.1968$ &  $0.4864$\\
    \hline
    $4$ & $0.1440$ & $0.4543$ & $0.1404$ &  $0.4872$ \\
     \hline
    $5$ & $0.1018$ & $0.5003$  & $0.1297$ & $0.1144$   \\
     \hline
    \end{tabular}
    \caption{$L_2$ error for $\frac{\partial u_1 }{\partial x}$ on domain $\Omega_2$.}
    \label{L2 sharingan adjacency error 1}
\end{table}

\begin{table}
    \centering
    \renewcommand{\arraystretch}{1.5} 
    \begin{tabular}{|P{1.3cm}|c|c|c|c|c|}
    \hline
    \multicolumn{1}{|c|}{Grid} & \multicolumn{2}{c|}{Euclidean Radius}  & \multicolumn{2}{c|}{Delaunay Deg 2}  \\
    \hline
    $1$ & $0.1622$ &  & $0.1818$ &\\
    \hline
    $2$ & $0.1000$ & $0.6978$ & $0.1259$ & $0.5301$ \\
    \hline
    $3$ & $0.07470$ & $0.4208$ & $0.09318$ & $0.4342$\\
    \hline
     $4$ & $0.05211$ & $0.5195$ & $ 0.06450$ & $0.5307$  \\
     \hline
     $5$ & $0.03841$ & $0.4400$  & $ 0.07840$ & $-0.2816$  \\
     \hline
    \end{tabular}
    \caption{$L_2$ error for $\frac{\partial u_2 }{\partial x}$ on domain $\Omega_2$.}
    \label{L2 sharingan adjacency error 2}
\end{table}

\noindent From Tables \ref{L2 sharingan adjacency error 1} and \ref{L2 sharingan adjacency error 2}, we see that the Euclidean Radius and Delaunay degree 2 method both gives similar errors and convergence rates for the first four grid sizes. However, for the densest grid, the Euclidean Radius continues to converge at a rate of $\frac{1}{2}$ while the errors for the Delaunay degree 2 method essentially plateau. Hence, all numerical results presented later in this paper will use the Euclidean Radius notion of adjacency to build meshfree SBP differentiation operators. \\

\subsection{Details on the computational implementation} \label{Notes on Implementation}
    \noindent For large sets of points in $S$, care must be taken to construct the SBP matrices in an efficient manner. For example, from Table \ref{L2 error simple_1} and \ref{L2 error simple_2}, there are $1132984$ points in $S$ for the $n_x = n_y = 1200, n_b = 4000$ case (Grid 5). Hence $\fnt{Q}_x, \fnt{Q}_y, \fnt{E}_x, \fnt{E}_y, \fnt{S}_x, \fnt{S}_y, \fnt{L}$, and $\fnt{D}_x, \fnt{D}_y$ are all $1132984$ by $1132984$ matrices. 
    We discuss a few implementation details used in our numerical experiments.     The full code can be found in the reproducibility repository \cite{} 
    In addition to using vectorized indexing and efficient native MATLAB matrix operations, we used the following additional  steps to reduce the computational cost of computing SBP matrices: 

    \begin{enumerate}

        \item  All matrices were constructed in MATLAB using various built-in packages and functions. Sparse matrices were utilized to reduce memory and computational costs. 

        \item In order to effectively calculate $\fnt{L}$, we construct it given the adjacency matrix $\fnt{A}$. However, constructing $\fnt{A}$ using a naive implementation of the Euclidean Radius method (e.g., create a distance matrix using \verb$pdist$ in MATLAB) is infeasible for large point sets. Instead, we construct the adjacency matrix using the Euclidean distance K-d tree implementation \verb$KDTreeSearcher$ in Matlab \cite{KDTreeSearcher}.

        \item Because the Delaunay triangulation sometimes connects nodes across boundaries of the interior circles in $\Omega_2$, the resulting errors are very large at the boundaries. We ensure this does not happen for the Delaunay Degree 2 triangulation method by manually removing connections between nodes situated across interior circle boundaries. 
        
        \item As mentioned in section \ref{subsec: testing H}, before solving \eqref{quadratic op} is computationally expensive for large point sets. We use a splitting conic solver to reduce the computational cost of solving \eqref{quadratic op} \cite{scs}.
        
    \end{enumerate}
    
    \noindent Table~\ref{Algo time} shows the computational runtimes of several steps involved in generating the SBP differentiation operators for the $n_x = n_y = 1200, n_b = 4000$ case. Notice that the majority of the time was spent constructing $\fnt{\Psi}_x $ and $\fnt{\Psi}_y$. This was because a large $1132984$ by $1132984$ matrix system is solved in equation \ref{systems of equations}. This can be potentially sped up with a more efficient numerical method (we utilize the default solver implemented in MATLAB's ``backslash''). For example, one can invert \textbf{L} efficiently using an algebraic multigrid solver \cite{Inverting_L}. 

    \begin{table}
    \centering
    \caption{Time taken for each matrix generation of the SBP  matrices (seconds)}
    \begin{tabular}{|c|c|c|c|c|c|c|c|}
    \hline
     $\fnt{A}$ & $\fnt{L}$ & $\fnt{E}_x, \fnt{E}_y$ & $\fnt{\fnt{\Psi}}_x, \fnt{\fnt{\Psi}}_y$ & $\fnt{S}_x, \fnt{S}_y$ & $\fnt{Q}_x, \fnt{Q}_y$ & $\fnt{H}_{\rm opt}$ & Total Time \\
    \hline 
    $5.0217$ & $0.9286$ & $0.0128$ & $163.9192$ & $8.8232 $ & $0.3200$ & $5.5201$ & $189.5627$ \\
    
     \hline
    \end{tabular}
     
    \label{Algo time}
\end{table}

\section{A robust first order accurate meshfree method} 
\label{sec: numerical method}

\noindent In this section, we apply our meshfree SBP operators to the numerical solution of systems of nonlinear conservation laws:
\begin{gather*}
        \frac{\partial \bm{u}}{\partial t}  + \frac{\partial \bm{f}_x(\bm{u})}{\partial x} + 
        \frac{\partial \bm{f}_y(\bm{u})}{\partial y} = \bm{0},\\ \bm{u}(x, y) = \bm{f}_{bc}(x, y), \quad \forall (x,y) \in M \cap S \label{nonlinear conservation law}
  \end{gather*}
  Here, $\bm{u} \in [L^2(\Omega)]^n$, $\bm{f}_x, \bm{f}_y, \bm{f}_{bc} \in [C^0(\mathbb{R}^n)]^n$, and $M$ is an appropriate subset of $\partial \Omega$ on which we enforce boundary conditions. 

 \subsection{Notation}
  \noindent To begin, we introduce some notation. Recall the notation used in \eqref{nonlinear conservation law}, where $\bm{u} \in [L^2(\Omega)]^n$ denotes the solution we are trying to approximate, where $\bm{u} = [u^1, u^2, ...\,, u^n]$. Here, $u^i$ denotes individual components of the solution, where $u^1, ... \, , u^n \in L^2(\Omega)$. We denote the numerical approximation of point values of $\bm{u}$ evaluated on $(x_i, y_i) \in S$ as:

  \begin{equation}
      \accentset{\rightharpoonup}{\fnt{u}} = 
      \renewcommand{\arraystretch}{1.5} 
    \begin{pmatrix} 
        [u^1(x_1, y_1)\, , ... \,, u^n(x_1, y_1)]  \\[1em] %
        [u^1(x_2, y_2)\, , ... \,, u^n(x_2, y_2)]  \\[1em] 
                
        \vdots \\[0.5em] 

        [u^1(x_N, y_N)\, , ... \,, u^n(x_N, y_N)]
    \end{pmatrix} 
  \end{equation}

  \noindent We denote $\accentset{\rightharpoonup}{\fnt{u}}_i$ as the $i^{th}$ entry of  $\accentset{\rightharpoonup}{\fnt{u}}$. Hence $\accentset{\rightharpoonup}{\fnt{u}}_i = [u^1(x_i, y_i)\, , ... \,, u^n(x_i, y_i)]$. Recall again from \eqref{nonlinear conservation law} that $\bm{f}_x: \mathbb{R}^n\rightarrow \mathbb{R}^n$, where $\bm{f}_x  = [f_x^1, f_x^2 \, , ... \, , f_x^n]$ and $f_x^1\, , ... \, , f_x^n \in C^0(\mathbb{R}^n)$. \\

    \newcommand{\halfvec}[1]{\accentset{\rightharpoonup}{#1}} 

    \noindent We now define the behavior of $\bm{f}_x(\halfvec{\fnt{u}})$ (e.g., $\bm{f}_x(\bm{u})$ evaluated over $S$) as: 
  
   \begin{equation}
      \renewcommand{\arraystretch}{1.5} 
      \bm{f_x}(\halfvec{\fnt{u}}) = 
    \begin{pmatrix} 
        \bm{f}_x(\halfvec{\fnt{u}}_1)  \\[1em] %
        \bm{f}_x(\halfvec{\fnt{u}}_2)  \\[1em]

        \vdots \\[0.5em] 

        \bm{f}_x(\halfvec{\fnt{u}}_N)  \\
    \end{pmatrix} = \begin{pmatrix} 
        [f_x^1(\halfvec{\fnt{u}}_1)\, , ... \, , f_x^n(\halfvec{\fnt{u}}_1)]  \\[1em] %
        
        [f_x^1(\halfvec{\fnt{u}}_2)\, , ... \, , f_x^n(\halfvec{\fnt{u}}_2)]  \\[1em] %

        \vdots \\[0.5em] 

         [f_x^1(\halfvec{\fnt{u}}_N)\, , ... \, , f_x^n(\halfvec{\fnt{u}}_N)]  \\[1em] %
    \end{pmatrix} 
  \end{equation}

  \noindent Furthermore, consider the matrix-vector multiplication $\fnt{D}_x\bm{f}_x(\halfvec{\fnt{u}})$. We define this behavior as follows:
\begin{equation}
    \fnt{D}_x\bm{f}_x(\halfvec{\fnt{u}}) = 
    \renewcommand{\arraystretch}{1.5} 
    \begin{pmatrix}
         \sum_{i = 1}^N  (\fnt{D}_{x})_{1i} \bm{f}_x(\halfvec{\fnt{u}}_i) \\[1em] 

         \sum_{i = 1}^N  (\fnt{D}_{x})_{2i} \bm{f}_x(\halfvec{\fnt{u}}_i) \\[1em] 
        
        \vphantom{\fnt{D}_x} \vdots \\[0.5em] 

       \sum_{i = 1}^N  (\fnt{D}_{x})_{Ni} \bm{f}_x(\halfvec{\fnt{u}}_i) \\[1em] 
    \end{pmatrix} = \begin{pmatrix} 
          \sum_{i = 1}^N  (\fnt{D}_{x})_{1i}[f_x^1(\halfvec{\fnt{u}}_i)\, , ... \, , f_x^n(\halfvec{\fnt{u}}_i)]  \\[1em] %
        
          \sum_{i = 1}^N  (\fnt{D}_{x})_{2i}[f_x^1(\halfvec{\fnt{u}}_i)\, , ... \, , f_x^n(\halfvec{\fnt{u}}_i)]  \\[1em] %

        \vdots \\[0.5em] 

           \sum_{i = 1}^N  (\fnt{D}_{x})_{Ni}[f_x^1(\halfvec{\fnt{u}}_i)\, , ... \, , f_x^n(\halfvec{\fnt{u}}_i)]  \\[1em] %
    \end{pmatrix} 
\end{equation}

  \subsection{Discretization}

  \noindent Inserting our meshfree SBP differentiation operators into \eqref{nonlinear conservation law} yields:
  \begin{equation}
        \fnt{\halfvec{u}}_t + \fnt{D}_x\fnt{f}_x(\fnt{\halfvec{u}}) + \fnt{D}_y\fnt{f}_y(\fnt{\halfvec{u}}) = \bm{0} 
  \end{equation}
  Multiplying $\fnt{H}$ to both sides yields:
  \begin{equation}
       \fnt{H}\fnt{\halfvec{u}}_t + \fnt{Q}_x\fnt{f}_x(\fnt{\halfvec{u}}) + \fnt{Q}_y\fnt{f}_y(\fnt{\halfvec{u}}) = 0 
  \end{equation}
Using \eqref{Q-E property}, we have that
    \begin{equation}
        \fnt{Q}_x\fnt{f}_x(\fnt{\halfvec{u}}) = \fnt{E}_x \fnt{f}_x(\fnt{\halfvec{u}}) - \fnt{Q}_x^T\fnt{f}_x(\fnt{\halfvec{u}}).
        \label{Q_BC_equation}
    \end{equation}

\noindent We make the observation that the term $\fnt{E}_x \fnt{f}_x(\fnt{\halfvec{u}})$ only depends on the points in $S \cap \partial \Omega$. Therefore we make the substitution $\halfvec{\fnt{u}}$ with $\halfvec{\fnt{u}}_{BC}$ where:
\begin{equation}
    \forall (x_i, y_i) \in S: 
    (\halfvec{\fnt{u}}_{BC})_{i}  = 
                \left\{
                    \begin{array}{lr}
                     \bm{f}_{bc}(x_i, y_i) & \text{if $(x_i, y_i) \in M$} \\
                     \fnt{u}_i & \text{otherwise} 
                    \end{array}
                \right. \label{u_BC}
\end{equation}

\noindent This allows us to weakly enforce the boundary conditions by replacing the right hand side of (\ref{Q_BC_equation}) by the following
\begin{equation}
      \fnt{E}_x \fnt{f}_x(\fnt{u}_{BC}) - \fnt{Q}_x^T\fnt{f}_x(\fnt{u}). 
\end{equation}

\noindent Using the SBP property \eqref{Q-E property} again yields:
\begin{equation}
    \fnt{E}_x \fnt{f}_x(\halfvec{\fnt{u}}_{BC}) - \fnt{Q}_x^T\fnt{f}_x(\halfvec{\fnt{u}}) = \fnt{Q}_x \fnt{f}_x(\halfvec{\fnt{u}}) + \fnt{E}_x({\fnt{f}_x(\halfvec{\fnt{u}}_{BC}) - \fnt{f}_x(\halfvec{\fnt{u}})})
\end{equation}

\noindent Together, this results in an algebraic formulation with weakly imposed boundary conditions:
\begin{equation}
    \fnt{H}\halfvec{\fnt{u}}_t + \fnt{Q}_x \fnt{f}_x(\halfvec{\fnt{u}}) + \fnt{Q}_y \fnt{f}_y(\halfvec{\fnt{u}}) + \fnt{E}_x({\fnt{f}_x(\halfvec{\fnt{u}}_{BC}) - \fnt{f}_x(\halfvec{\fnt{u}})})  + \fnt{E}_y({\fnt{f}_y(\halfvec{\fnt{u}}) - \fnt{f}_y(\halfvec{\fnt{u}})})= \bm{0} \label{weak boundary formulation} 
\end{equation}

\subsection{Stabilization}

\noindent The formulation \eqref{weak boundary formulation} corresponds to a non-dissipative ``central'' scheme. However, for solutions of nonlinear conservation laws with sharp gradients, shocks, or other under-resolved solution features, non-dissipative schemes can result in spurious oscillations \cite{leveque1992numerical, leveque2002finite}. To avoid this, we add upwinding-like dissipation using an approach similar to techniques used in \cite{kwan2012conservative}. 

\noindent Denote $\bm{r}_i$ as the $i^{\text{th}}$ entry of $\fnt{Q}_x \fnt{f}_x(\halfvec{\fnt{u}}) + \fnt{Q}_y \fnt{f}_y(\halfvec{\fnt{u}})$. Then, we observe the following:
\begin{align}
    \bm{r}_i &= \sum_{j} (\fnt{Q}_x)_{ij}\fnt{f}_x(\halfvec{\fnt{u}}_j) + (\fnt{Q}_y)_{ij}\fnt{f}_y(\halfvec{\fnt{u}}_j) \label{r_i}\\ 
     &=  2 \sum_{j} (\fnt{Q}_x)_{ij}\frac{(\fnt{f}_x(\halfvec{\fnt{u}}_j) + \fnt{f}_x(\halfvec{\fnt{u}}_i))}{2} + (\fnt{Q}_y)_{ij}\frac{(\fnt{f}_y(\halfvec{\fnt{u}}_j) + \fnt{f}_y(\halfvec{\fnt{u}}_i))}{2} \label{r_i 2}
\end{align}

\noindent Note that moving from \eqref{r_i} to \eqref{r_i 2} is valid because $\fnt{Q_x} \bm{1} = \bm{0}$ by the first order consistency condition, which implies that $\sum_{j} (\fnt{Q}_x)_{ij}  = 0$. Because multiplying both sides by $\fnt{f}_x(\halfvec{\fnt{u}}_i)$ still yields $\bm{0}$, $\sum_{j} (\fnt{Q}_x)_{ij} \fnt{f}_x(\halfvec{\fnt{u}}_i) = \sum_{j} (\fnt{Q}_y)_{ij} \fnt{f}_y(\halfvec{\fnt{u}}_i) = \bm{0}$, implying that \eqref{r_i 2} is equivalent to \eqref{r_i}. \\

\noindent  We can rewrite \eqref{r_i 2} as 
\begin{align}
\bm{r}_i(\halfvec{\fnt{u}}) = 2 \sum_{j} \frac{1}{2} \left(\begin{bmatrix} 
\fnt{f}_x(\halfvec{\fnt{u}}_i) \\
\fnt{f}_y(\halfvec{\fnt{u}}_i)  
\end{bmatrix} + \begin{bmatrix} 
\fnt{f}_x(\halfvec{\fnt{u}}_j) \\
\fnt{f}_y(\halfvec{\fnt{u}}_j)  
\end{bmatrix}\right) \cdot \begin{bmatrix} 
    (\fnt{Q}_x)_{ij} \\
(\fnt{Q}_y)_{ij}  
\end{bmatrix} \label{r_i 3}
\end{align}

\noindent We make the observation that \eqref{r_i 3} is similar to a central flux, where $(\fnt{Q}_x)_{ij}, (\fnt{Q}_y)_{ij}$ are playing the role of the ``normal vector''. However, because the central flux is non-dissipative, we replace the central flux with the local Lax Friedrichs flux in our formulation as follows:
\begin{equation}
    \bm{r}_i(\halfvec{\fnt{u}}) = 2 \sum_{j \in N(i)} \|\bm{n}_{ij}\| \frac{1}{2} \left(\begin{bmatrix} 
\fnt{f}_x(\halfvec{\fnt{u}}_i) \\
\fnt{f}_y(\halfvec{\fnt{u}}_i)  
\end{bmatrix} + \begin{bmatrix} 
\fnt{f}_x(\halfvec{\fnt{u}}_j) \\
\fnt{f}_y(\halfvec{\fnt{u}}_j)  
\end{bmatrix}\right) \cdot \frac{\bm{n}_{ij}}{\|\bm{n}_{ij}\|} - \frac{\lambda}{2}(\halfvec{\fnt{u}}_j - \halfvec{\fnt{u}}_i), \label{r_i 4}
\end{equation}

\noindent where we use the following definition for the algebraic ``normal'' vector $\bm{n}_{ij}$:
\begin{equation}
\label{eq:nij}
\bm{n}_{ij} = \begin{bmatrix} 
(\fnt{Q}_x)_{ij} \\
(\fnt{Q}_y)_{ij} \end{bmatrix}.
\end{equation}

\noindent Finally, as noted in \cite{kwan2012conservative}, one can also substitute a general Riemann solver into this formulation by introducing the numerical flux $\bm{f}(\bm{u}_L, \bm{u}_R, \bm{n}_{LR})$
\begin{equation}
    \bm{r}_i(\halfvec{\fnt{u}}) = 2 \sum_{j \in N(i)} \|\bm{n}_{ij}\| \bm{f}(\halfvec{\fnt{u}}_i, \halfvec{\fnt{u}}_j, \bm{n}_{ij})
    \label{r_i 5}
\end{equation}
For example, taking the numerical flux to be the local Lax-Friedrichs flux 
\[
\bm{f}(\bm{u}_L, \bm{u}_R, \bm{n}_{LR}) = \frac{\bm{f}(\bm{u}_L) + \bm{f}(\bm{u}_R)}{2}\bm{n}_{LR} - \frac{\lambda}{2}(\bm{u}_R - \bm{u}_L)
\]
recovers formulation \eqref{r_i 4}.\\

\begin{remark}
\noindent If $\lambda$ is an upper bound on the maximum wave speed of the system, then \eqref{r_i 5} recovers a semi-discrete version of the invariant domain preserving discretization of \cite{guermond2019invariant}, at least at interior points. This implies that, for a sufficiently small time-step, the numerical solution will preserve e.g., positivity of density and internal energy. We note that invariant domain preservation is only guaranteed if the numerical flux $\bm{f}(\bm{u}_L, \bm{u}_R, \bm{n}_{LR})$ is the local Lax-Friedrichs flux, and does not hold for more general positivity-preserving numerical fluxes (e.g., the HLLC flux \cite{batten1997choice, toro2019hllc}). \\
\end{remark}

\noindent Using \eqref{r_i 4} or \eqref{r_i 5}, we can simplify our formulation in \eqref{weak boundary formulation} to 
\begin{equation}
    \halfvec{\fnt{u}}_t  =  \textbf{rhs}(\halfvec{\fnt{u}}) \label{final formulation}
\end{equation}

\noindent where \textbf{rhs(u)} is computed by Algorithm \ref{algo 1}:  \\

\begin{algorithm} 
    \caption{RHS Function}
    \label{Algorithm 1}
    \begin{algorithmic}[1]
        \State $\text{Inputs: } \halfvec{\fnt{u}}, \fnt{H}, \fnt{Q}_x, \fnt{Q}_y, \bm{f}_x, \bm{f}_y, \bm{f}_{bc}, \bm{x}_{bc}, \bm{y}_{bc}, \bm{f}, t$
        \State $\halfvec{\fnt{u}}_{bc} \gets$ as defined in \eqref{u_BC}
        \For{$(i, j) \text{ s.t.} Q_{ij} \neq 0$}
            \State $\bm{n}_{ij} \gets  \begin{bmatrix} 
            (\fnt{Q}_x)_{ij} \\
            (\fnt{Q}_y)_{ij}  
            \end{bmatrix}$
            \State $\fnt{du}[i] \gets \fnt{du}[i] + \|\bm{n}_{ij}\| 
            \bm{f}\left(\halfvec{\fnt{u}}_i, \halfvec{\fnt{u}}_j, \frac{\bm{n}_{ij}}{\|\bm{n}_{ij}\|}\right)$
        \EndFor
    \State $\fnt{du} \gets \fnt{du} + \fnt{E}_x(\fnt{f}_x(\halfvec{\fnt{u}}_{bc}) - \fnt{f}_x(\halfvec{\fnt{u}})) + \fnt{E}_y(\fnt{f}_y(\halfvec{\fnt{u}}_{bc}) - \fnt{f}_y(\halfvec{\fnt{u}}))$
    \State $\textbf{rhs(u)} \gets -\fnt{H}^{-1}\fnt{du}$
    \end{algorithmic} \label{algo 1}
    \end{algorithm}

\noindent Finally, to numerically integrate \eqref{final formulation}, one can use any suitable time-stepping method. Unless stated otherwise, we utilize the 4-stage 3rd order Strong Stability Preserving (SSP) Runge-Kutta method \cite{kraaijevanger1991contractivity, conde2018embedded, ranocha2022optimized}. All numerical results utilize the DifferentialEquations.jl library \cite{rackauckas2017differentialequations}, the Trixi.jl library \cite{trixi_1}, \cite{trixi_2}, and are implemented in the Julia programming language \cite{bezanson2017julia}.

\section{Numerical results} \label{sec: results}
\noindent We apply the numerical method described in Section~\ref{sec: numerical method} to the advection equation and compressible Euler equations. We begin by analyzing numerical rates of convergence in the discrete $L^2$ norm for analytical solutions. 

\subsection{Advection Equation}
\noindent Consider the advection equation with the following boundary and initial conditions:
\begin{align}
     \frac{\partial u}{\partial t} + a\frac{\partial u}{\partial x} + b\frac{\partial u}{\partial y}= 0 \label{advection equation}, \\ 
     u(x, y, t) = u_{bc}(x, y, t) \quad \forall (x, y) \in \Gamma_{\rm in}, \label{advection equation bc}\\ 
     u(x, y, 0) = u_0(x, y), \label{advection equation initial conditions}
\end{align}
where we have defined the inflow boundary $\Gamma_{\rm in} = \partial \Omega \; \;  s.t. \; [a, b] \cdot [n_x, n_y] < [0, 0]$. \\

\noindent We test the numerical method on \eqref{advection equation}-\eqref{advection equation initial conditions} with $a, b = 1, \frac{1}{2}$, final time $t_f = 0.7$, and exact solution
\begin{equation}
    u(x, y, t) = \sin\left(\frac{\pi}{6}(x-t)\right)\sin\left(\frac{\pi}{6}(y - \frac{1}{2}t)\right) 
\end{equation}
We impose the inflow boundary by setting $\bm{u}_{bc}$ to the exact solution. Table~\ref{L2 error advection} displays the $L^2$ errors on domains $\Omega_1$ and $\Omega_2$. Grids 1, \ldots, 5 refer to the grids used in Section~\ref{sec: finding adjacency} to compute approximation errors under different methods for computing the adjacency matrix).\\
    
\begin{table}
\centering
\renewcommand{\arraystretch}{1.5} 
\begin{tabular}{|P{1.3cm}|c|c|c|c|c|}
\hline
\multicolumn{1}{|c|}{Grid} & \multicolumn{1}{c|}{$L^2$ errors on $\Omega_1$} & \multicolumn{1}{c|}{Rate} & \multicolumn{1}{c|}{$L^2$ errors on $\Omega_2$} & \multicolumn{1}{c|}{Rate}  \\
\hline
$1$ & $0.06515$ &  & $0.06825$ &  \\
\hline
$2$ & $0.03334$ & $0.9665$ & $0.03524$ & $0.9536$ \\
\hline
$3$ & $0.01660$ & $1.006$ & $0.01825$ & $0.9493$ \\
\hline
 $4$ & $0.008426$ & $0.9782$ & $0.009448$ & $0.9498$ \\
 \hline
 $5$ & $0.004177$ & $1.012$ & $0.004720$ & $1.177$  \\
 \hline
\end{tabular}
\caption{$L^2$ errors for the advection equation on the domains $\Omega_1$ and $\Omega_2$.}
\label{L2 error advection}
\end{table}
\noindent As seen in both cases, the $L_2$ error converges with a rate of $1$ despite having only a convergence rate of $0.5$ in the differentiation experiments in Tables \ref{L2 error simple_1} and \ref{L2 error simple_2}.



\subsection{Compressible Euler Equations}
\noindent We now consider the 2D compressible Euler equations:
\begin{equation}
    \frac{\partial}{\partial t} \left( \begin{matrix} \rho \\ \rho v_1 \\ \rho v_2 \\ \rho e \end{matrix} \right) + 
    \frac{\partial}{\partial x} \left( \begin{matrix} \rho v_1 \\ \rho v_1^2 + p \\ \rho v_1 v_2 \\ (\rho e + p)v_1 \end{matrix} \right) + 
    \frac{\partial}{\partial y} \left( \begin{matrix} \rho v_2 \\ \rho v_1 v_2 \\ \rho v_2^2 + p\\ (\rho e + p)v_2 \end{matrix} \right) =
    \left( \begin{matrix} 0 \\ 0 \\ 0 \\ 0\end{matrix} \right),
\end{equation}
\noindent where $\rho$ is the density, $v_1, v_2$ are the velocities, and $e$ is the specific total energy. The pressure $p$ is given by the ideal gas law
\begin{equation}
    p = (\gamma - 1)\left(\rho e - \frac{1}{2}\rho(v_1^2 + v_2^2)\right)
\end{equation}
\noindent where $\gamma = 1.4$ is the ratio of the specific heats. \\

\noindent We test the numerical method on the density wave solution:
\begin{equation}
    v_1 = 0.1, \quad v_2 = 0.2, \quad \rho = 1 + \frac{1}{2}\sin\left(\frac{1}{3}(x + y - t(v_1 + v_2))\right), \quad p = 2.5 
\end{equation}

\noindent Table \ref{L2 error euler lax} displays $\| u(t_f) - \hat{u}(t_f)\|_{L^2(\Omega)}$ and $\| u(t_f) - \hat{u}(t_f)\|_{L^2(\Omega_2)}$ at final time $t_f = 0.7$. We impose reflective slip wall boundary conditions on all domains using a similar technique mentioned in \cite{Vegt2002SlipFB}.\\

\begin{table}
    \centering
    \renewcommand{\arraystretch}{1.5} 
    \begin{tabular}{|P{1.3cm}|c|c|c|c|c|}
    \hline
    \multicolumn{1}{|c|}{Grid} & \multicolumn{1}{c|}{$L^2$ errors on $\Omega_1$} & \multicolumn{1}{c|}{Rate} & \multicolumn{1}{c|}{$L^2$ errors on $\Omega_2$} & \multicolumn{1}{c|}{Rate}  \\
    \hline
    $1$ & $0.3979$ & & $0.3102$ &  \\
    \hline
    $2$ & $0.2168$ & $0.8754$   & $0.1777$  &  $0.8037$ \\
    \hline
    $3$ & $0.1142$ & $0.9428$ & $0.09691$ & $0.8747$\\
    \hline
     $4$ & $0.05889$ & $0.9555$ & $0.05112$ & $0.9227$ \\
     \hline
     $5$ & $0.03000$ & $0.9731$  & $0.02641$ &  $0.9528$  \\
     \hline
    \end{tabular}
    \caption{$L^2$ errors for the compressible Euler equations under the local Lax-Friedrichs flux.} 
    \label{L2 error euler lax}
\end{table}

\noindent We first consider using the local Lax-Friedrichs flux with a Davis wavespeed estimate \cite{davis1988simplified}. We observe in Table~\ref{L2 error euler lax} that the convergence rate approaches one. However, the magnitude of the error is an order of magnitude larger than the errors reported for the advection equation in Table~\ref{L2 error advection}. We believe this to be due to the highly dissipative nature of the local Lax-Friedrichs flux. \\

\noindent To confirm this, we investigate the HLLC (Harten-Lax-van Leer contact) flux  \cite{HLLC_paper}, which is known to be less dissipative than the local Lax-Friedrichs flux. We replace the Lax-Friedrich Flux with the HLLC flux in algorithm \ref{algo 1} and show the results for $\| u(t_f) - \hat{u}(t_f)\|_{L^2(\Omega)}$ and $\| u(t_f) - \hat{u}(t_f)\|_{L^2(\Omega_2)}$ in Table~\ref{L2 error euler HLLC}. Under the HLLC flux, the convergence rate remains near one, but the magnitude of error is now similar to the magnitude of errors observed for the advection equation in Table~\ref{L2 error advection}. 
\\

\begin{table}
    \centering
    \renewcommand{\arraystretch}{1.5} 
    \begin{tabular}{|P{1.3cm}|c|c|c|c|c|}
    \hline
    \multicolumn{1}{|c|}{Grid} & \multicolumn{1}{c|}{$L^2$ errors on $\Omega_1$} & \multicolumn{1}{c|}{Rate} & \multicolumn{1}{c|}{$L^2$ errors on $\Omega_2$} & \multicolumn{1}{c|}{Rate}  \\
    \hline
    $1$ & $0.05017$ & & $0.05071$ &  \\
    \hline
    $2$ & $0.02536$ & $0.9842$ & $0.02580$  & $0.9749$  \\
    \hline
    $3$ & $0.01253$ & $1.0171$ & $0.01277$ & $1.0146$ \\
    \hline
     $4$ & $0.006210$ & $1.0127$ & $0.006323$ & $1.0141$ \\
     \hline
     $5$ & $0.003069$ & $1.0168$ & $0.003114$ &  $1.0218$  \\
     \hline
    \end{tabular}
    \caption{$L^2$ errors for the compressible Euler equations under the HLLC flux.}    
    \label{L2 error euler HLLC}
\end{table}

\noindent Finally, we notice that there is not a significant difference in the errors achieved on $\Omega_1$ and $\Omega_2$. This  differs from what we observed when computing $L^2$ errors for the approximation of derivatives using SBP operators; in those experiments, the errors were larger on domain $\Omega_2$. \\

\noindent Next, we test the numerical method on a\eqref{advection equation}-\eqref{advection equation initial conditions} with a $C^0$ solution: 
\begin{equation}
    v_1, v_2 = [0.1, 0.2] \quad \rho = 1 + \frac{1}{2}\left| \sin\left(\frac{1}{3}(x + y - t(v_1 + v_2))\right) \right| \quad p = 2.5 
\end{equation}
We again utilize parameters $a = 1, b = \frac{1}{2}$ and run to final time $t_f = 0.7$. Table~\ref{L2 error euler C0} shows the errors and computed convergence rates. We observe that the convergence rate appears to approach a value less than 1. \\ 

\begin{table}
    \centering
    \renewcommand{\arraystretch}{1.5} 
    \begin{tabular}{|P{1.3cm}|c|c|}
    \hline
    \multicolumn{1}{|c|}{Grid} & \multicolumn{1}{c|}{$\| u - \hat{u}\|_{L^2(\Omega)}$} & \multicolumn{1}{c|}{Convergence Rate} \\
    \hline
    $1$ & $0.4851$ &  \\
    \hline
    $2$ & $0.3107$ & $0.6428$  \\
    \hline
    $3$ & $0.1896$ & $0.7126$ \\
    \hline
     $4$ & $0.1131$ & $0.7454$  \\
     \hline
     $5$ & $0.06679$ & $0.7598$   \\
     \hline
    \end{tabular}
    \caption{$L^2$ errors for the compressible Euler equations under a $C^0$ continuous solution.}
    \label{L2 error euler C0}
\end{table}

\noindent We conclude by simulating an ``explosion'' problem, which is given by the following initial condition adapted from \cite{tominec2023residual}:
\begin{equation}
    v_1, v_2 = [0, 0] \quad \rho = \left\{
                \begin{array}{lr}
                1 & \text{if } x^2 + y^2 < 0.4^2 \\
                0.001 & \text{otherwise} 
                \end{array}
            \right\}  \quad p = \rho^\gamma
\end{equation}

\noindent Figure~\ref{fig:explosion_pressure} shows the pressure at various times computed using the HLLC flux in \eqref{r_i 5}. 

\begin{figure}
\centering
\includegraphics[width=.32\textwidth, trim={0 0 5em 0}, clip]{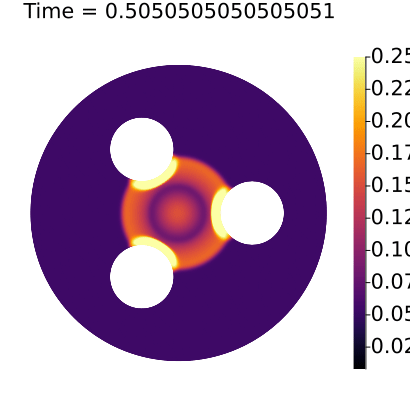}
\includegraphics[width=.32\textwidth, trim={0 0 5em 0}, clip]{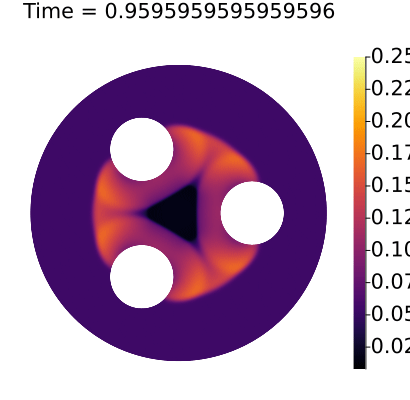}
\includegraphics[width=.32\textwidth, trim={0 0 5em 0}, clip]{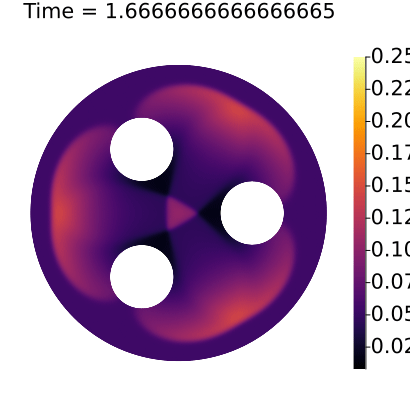}\\
\includegraphics[width=.32\textwidth, trim={0 0 5em 0}, clip]{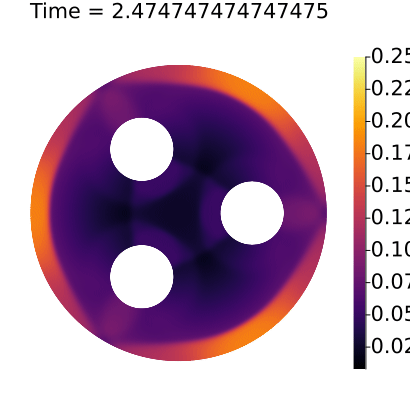}
\includegraphics[width=.32\textwidth, trim={0 0 5em 0}, clip]{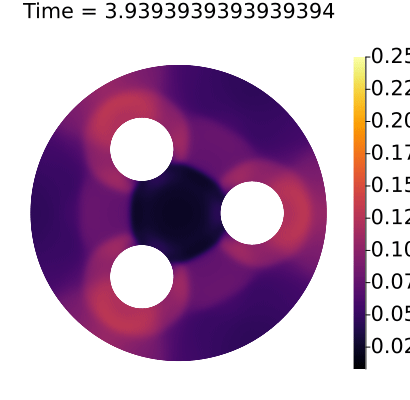}
\includegraphics[width=.32\textwidth, trim={0 0 5em 0}, clip]{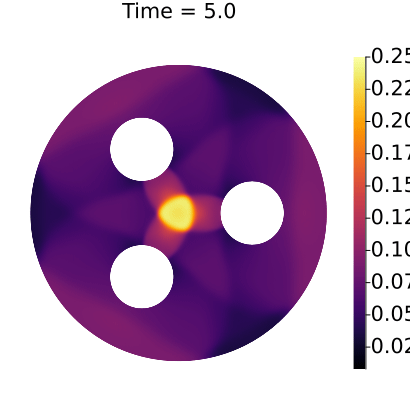}
\caption{The computed pressure using the HLLC flux at various times. The color limits are taken from 0.008 to 0.25.}
\label{fig:explosion_pressure}
\end{figure}

\noindent We additionally run the simulation until final time $t_f = 7$ and compare the solutions computed using the local Lax-Friedrichs and HLLC flux Figure~\ref{fig:explosion}. We observe that the HLLC flux produces less dissipative solutions compared to the local Lax-Friedrichs flux. 


\begin{figure}
    \centering   
    

    

    


    
    \begin{minipage}[b]{0.35\textwidth}
        \centering
        \includegraphics[width=\textwidth]{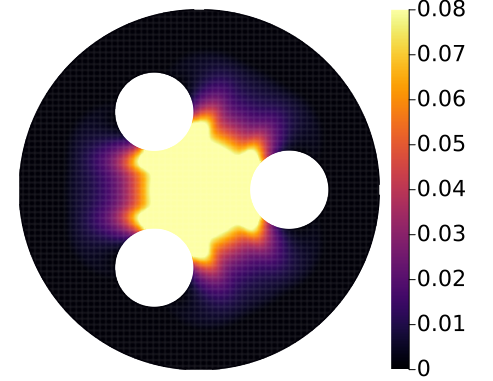}
        \subcaption{Lax-Friedrichs, $t = 0.583$}
    \end{minipage}
    \begin{minipage}[b]{0.35\textwidth}
        \centering
        \includegraphics[width=\textwidth]{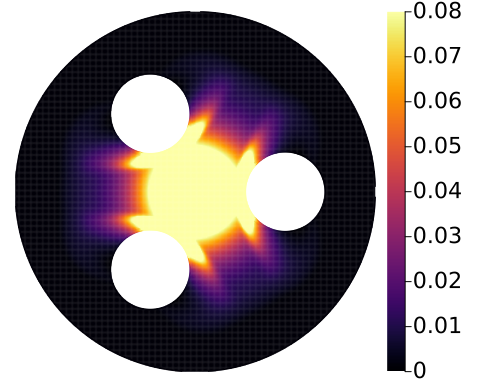}
        \subcaption{HLLC, $t = 0.583$}
    \end{minipage}
    
    
    
    \begin{minipage}[b]{0.35\textwidth}
        \centering
        \includegraphics[width=\textwidth]{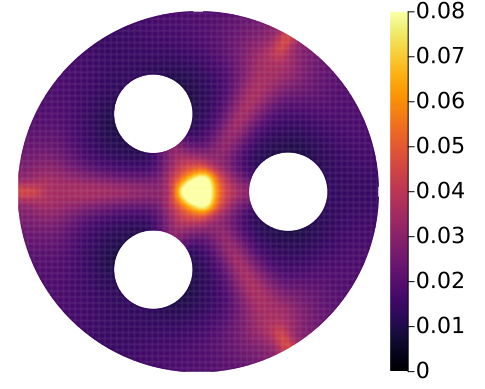}
        \subcaption{Lax-Friedrichs, $t = 4.67$}
    \end{minipage}
    \begin{minipage}[b]{0.35\textwidth}
        \centering
        \includegraphics[width=\textwidth]{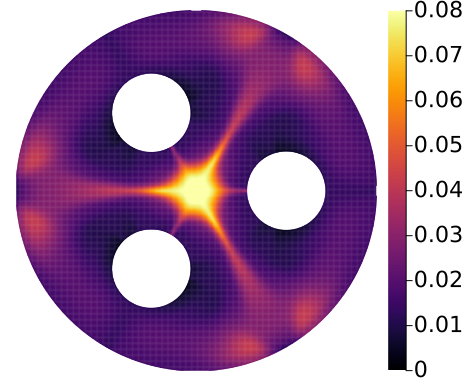}
        \subcaption{HLLC, $t = 4.67$}
    \end{minipage}
    
    \begin{minipage}[b]{0.34\textwidth}
        \centering
        \includegraphics[width=\textwidth]{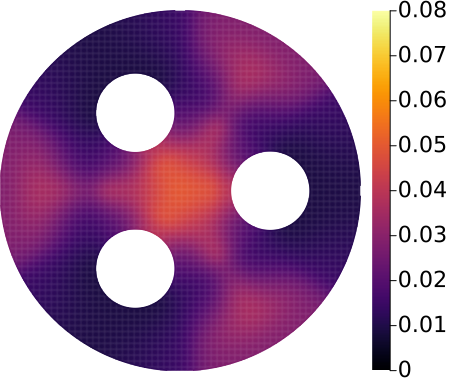}
        \subcaption{Lax-Friedrichs, $t = 5.54167$}
    \end{minipage}
    \begin{minipage}[b]{0.36\textwidth}
        \centering
        \includegraphics[width=\textwidth]{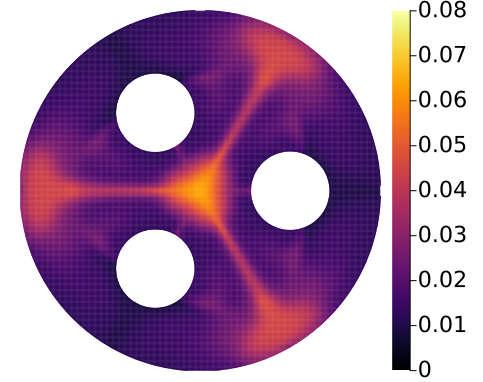}
        \subcaption{HLLC, $t = 5.54167$}
    \end{minipage}
    
    \caption{Comparison of the density $\rho$ at different times using Lax-Friedrichs and HLLC fluxes.}
    \label{fig:explosion}
\end{figure}

\section{Conclusion}
\noindent By using the concept of summation by parts, we were able to create first differentiations that did not require a mesh but only a notion of adjacency. Upon picking a suitable notion of adjacency, namely the Euclidean Radius, we were able to efficiently create such operators. Using these SBP operators, we create a mesh-free numerical method using a flux-based formulation to numerically approximate non-linear conservation laws. The method performs reasonably well and is often first order given sufficient smoothness.  \\

\noindent For reproducibility purposes, please refer to the GitHub repository in \cite{github} for the codebase used to generate the SBP operators and the numerical results.

\section{Acknowledgments}
\noindent The authors gratefully acknowledge support by the National Science Foundation under award NSF DMS-2231482. Chan was also supported by DMS-1943186.










\bibliographystyle{plain}
\bibliography{refs.bib}

\end{document}